\documentclass[twocolumn,8pt]{autart}    

\usepackage{amsmath,amssymb,enumerate,subfigure,multirow,hhline,color}
\usepackage{xcolor}
\usepackage{graphicx}
\usepackage{graphics}

\usepackage{epsfig,psfrag}
\usepackage{tabularx}
\usepackage{tabulary}
\usepackage{algorithm}
\usepackage{algpseudocode}

\usepackage[sort]{cite}

\usepackage{hyperref}
\hypersetup{breaklinks=true,colorlinks=true,linkcolor=blue,citecolor=blue}

\usepackage[noabbrev,capitalise,nameinlink]{cleveref}

\DeclareMathOperator{\argmin}{arg\,min}

\newtheorem{example}{Example}
\newtheorem{assumption}{Assumption}
\newtheorem{remark}{Remark}
\newtheorem{corollary}{Corollary}
\newtheorem{problem}{Problem}
\newtheorem{lemma}{Lemma}
\newtheorem{definition}{Definition}
\newtheorem{proposition}{Proposition}

\newenvironment{proof}{\begin{pf}}{\qed\end{pf}}

\allowdisplaybreaks

\begin{document}

\begin{frontmatter}

\title{Primal-Dual Distributed Temporal Difference Learning \thanksref{footnoteinfo}}

\thanks[footnoteinfo]{This paper was not presented at any IFAC meeting.}

\author[KAIST]{Donghwan Lee}\ead{donghwan@kaist.ac.kr},
\author[Purdue]{Jianghai Hu}\ead{jianghai@purdue.edu},

\address[KAIST]{Department of Electrical Engineering, KAIST, Daejeon, 34141, South Korea}
\address[Purdue]{Department of Electrical and Computer Engineering,
Purdue University, West Lafayette, IN 47906, USA}


\begin{abstract}                          
The goal of this paper is to study a distributed temporal-difference (TD)-learning algorithm for a class of multi-agent Markov decision processes (MDPs). The single-agent TD-learning is a reinforcement learning (RL)
algorithm to evaluate an accumulated rewards corresponding to a given policy. In multi-agent settings, multiple RL agents concurrently behave following its own local behavior policy and learn the accumulated global rewards, which is a sum of the local rewards. The goal of each agent is to evaluate the accumulated global rewards by only receiving its local rewards. The algorithm shares learning parameters through random network communications, which have a randomly changing undirected graph structures. The problem is converted
into a distributed optimization problem and the corresponding saddle-point problem of its Lagrangian function. The propose TD-learning is a stochastic primal-dual algorithm to solve it. We prove finite-time convergence of the algorithm with its convergence rates and sample complexity.
\end{abstract}

\end{frontmatter}

\section{Introduction}

We develop a new multi-agent
temporal-difference (TD)-learning
algorithm, called a distributed gradient temporal-difference (DGTD) learning, for multi-agent Markov decision processes (MDPs). TD-learning~\cite{sutton2009convergent,sutton2009fast} is a reinforcement learning (RL) algorithm to learn an accumulated discounted rewards for a given policy without the model knowledge, which is called the policy evaluation problem. In our multi-agent RL setting, $N$ RL agents concurrently behave and learn the accumulated global rewards, which is a sum of the local rewards, where each agent $i$ only receives local reward following its own local behavior policy $\pi_i$. The main challenge is the information limitation: each agent is only accessible to its local reward which only contains partial information on the global reward. The algorithm assumes additional partial information sharing among agents, e.g., sharing of learning parameters, through random network communications, where the network structure is represented by a randomly changing undirected graph. Despite the additional communication model, the algorithm is still distributed in the sense that each agent has a local view of the overall system: it is only accessible to the learning parameter of the neighboring RL agents in the graph. Potential applications are distributed machine learning, distributed resource allocation, and robotics, where the reward information is limited due to physical limitations (spacial limits in robotics or infrastructure limits in resource allocation) or privacy constraints.

The proposed DGTD generalizes the single-agent GTD~\cite{sutton2009convergent,sutton2009fast} to the multi-agent MDPs. The algorithm is derived according to the following steps: we cast the multi-agent policy evaluation problem as the distributed optimization problem
\begin{align}
&\min_{w^{(i)}} \sum_{i=1}^N {f_i(w^{(i)})}\quad {\rm subject\,\,to}\quad
w^{(1)}=w^{(2)}=\cdots=w^{(N)},\label{eq:distributed-optimization}
\end{align}
where $f_i$ is an objective of each agent, related to the Bellman loss function, and a corresponding single saddle-point optimization problem. The averaging consensus-based algorithms~\cite{jadbabaie2003coordination} are popular for solving the distributed optimization~\eqref{eq:distributed-optimization}. Different from the averaging consensus-based algorithms, the proposed DGTD applies the primal-dual saddle-point approach~\cite{wang2010control,wang2011control,gharesifard2014distributed,shi2015extra,mokhtari2016dsa,lei2016primal} to multi-agent RLs, where primal-dual algorithms are developed for the distributed optimization. Their main idea is to convert the constraints $w^{(1)}=w^{(2)}=\cdots=w^{(N)}$ in~\eqref{eq:distributed-optimization} into a single equality constraint with the graph Laplacian matrix, and solve the optimization by using the Lagrangian duality. It is known that they provide effective convergence rates with constant step-sizes for deterministic problems. We generalize it to stochastic cases using the stochastic primal-dual method~\cite{nemirovski2009robust}, and apply to the policy evaluation problem. Advantages of the primal-dual approach is that analysis tools from optimization perspectives, such as~\cite{chen2016stochastic,nedic2009subgradient,nemirovski2009robust,wang2010control,wang2011control,gharesifard2014distributed} can be easily applied to prove its convergence, and the case of random communication networks can be easily addressed.

The main contributions are summarized as follows:
\begin{enumerate}
\item To the author's knowledge, the proposed DGTD is the first multi-agent off-policy \footnote{The term ``off-policy'' means a property of RL algorithms, especially for the policy evaluation problem, that the behavior policy of the RL agent can be separated with the target policy we want to learn.} RL algorithm which guarantees convergence under distributed rewards. Only recently,~\cite{wai2018multi} and~\cite{ding2019fast} suggest multi-agent off-policy RLs at or after the time of initial submission of this paper. The differences are summarized shortly later.

\item This study provides a general and unified saddle-point framework of the distributed policy evaluation problem, which offers more algorithmic flexibility such as additional cost constraints and objective, for example, entropic measures and sparsity promoting objectives. In particular, we formalize the distributed policy evaluation problem as a distributed optimization, and then convert it into a single saddle-point problem. Another advantage of this approach is that it easily addresses the case of random communication networks.

\item  Rigorous analysis is given for the policy evaluation problem and the DGTD. In particular, we provide analysis of solutions of the proposed saddle-point problem including bounds on the solutions, and prove that the policy evaluation problem can be solved by addressing the saddle-point problem. We also provide rigorous convergence rates and sample complexity of the proposed algorithm, which are currently laking in the literature.
\end{enumerate}

{\em Related works:} Recently, some progresses have been made in multi-agent RLs~\cite{kar2013cal,zhang2018fully,macua2015distributed,stankovic2016multi,mathkar2017distributed}. For the policy optimization problem, the distributed Q-learning (QD-learning)~\cite{kar2013cal}, distributed actor-critic algorithm~\cite{zhang2018fully,suttle2019multi}, and distributed fitted Q-learning~\cite{zhang2018finite} are studied in multi-agent settings. The work in~\cite{qu2019value} considers an approximation distributed Q-learning with neural nonlinear function approximation. For the policy evaluation problem, distributed GTD algorithms are studied in~\cite{macua2015distributed,stankovic2016multi,doan2019convergence,wai2018multi,cassano2019distributed,ding2019fast,cassano2019distributed}. The results in~\cite{macua2015distributed,stankovic2016multi,cassano2019distributed} consider central rewards with different assumptions. The result in~\cite{doan2019convergence} suggests a distributed TD learning with an averaging consensus steps, and proves its convergence rate. The main difference is that~\cite{doan2019convergence} considers an on-policy learning, while this work considers off-policy learning methods. The TD learning in~\cite{wai2018multi} considers a stochastic primal-dual algorithm for the policy evaluation with stochastic variants of the consensus-based distributed subgradient method akin to~\cite{qu2017harnessing}. The main difference is that the algorithm in~\cite{wai2018multi} introduces gradient surrogates of the objective function with respect to the local primal and dual variables, and the mixing steps for consensus are applied to both the local parameters and local gradient surrogates. However, rigorous convergence analysis, such as the sample complexity and convergence with high probability, is lacking in~\cite{wai2018multi} compared to the work in this paper. The work in~\cite{ding2019fast} develops the so-called homotopy stochastic primal-dual algorithm with ${\mathcal O}(1/T)$ rate for strongly convex strongly concave min-max problems, where $T$ is the total number of iterations. The rate is faster than the rate of the proposed algorithm, ${\mathcal O}(1/\sqrt{T})$. However, the new algorithm can be applied to the proposed formulation and improve our result. Moreover, rigorous analysis of solutions is lacking in~\cite{ding2019fast}.

Preliminary results are included in the conference version~\cite{lee2018primal}, which only provides asymptotic convergence based on the stochastic approximation method~\cite{kushner2003stochastic} and control theory. However, the convergence without its rates and complexity analysis does not guarantee efficiency of the algorithm, which is essential in contemporary optimization and learning algorithms. The convergence rate analysis is usually more challenging and requires substantially more works. In this paper, we provide more rigorous and comprehensive analysis of solutions and finite-time convergence rate analysis with sample complexities based on results in convex optimization, which is not possible in the control theoretic approach in~\cite{lee2018primal}. Besides, we consider stochastic network communications and a modified algorithm to improve its convergence properties.

\section{Preliminaries}
\subsection{Notation and terminology}
The following notation is adopted: ${\mathbb R}^n $ denotes the $n$-dimensional Euclidean space; ${\mathbb R}^{n \times m}$ denotes the set
of all $n \times m$ real matrices; ${\mathbb R}_+ $ and ${\mathbb
R}_{++}$ denote the sets of nonnegative and positive real numbers,
respectively, $A^T$ denotes the transpose of matrix $A$; $I_n$ denotes the $n \times n$
identity matrix; $I$ denotes the identity matrix with  appropriate
dimension; $\|\cdot \|_2$ denotes the standard Euclidean norm; $\|x\|_D:=\sqrt{x^T Dx}$ for any
positive-definite $D$; $\lambda_{\min}(A)$ denotes the minimum eigenvalue of
$A$ for any symmetric
matrix $A$; $|{\mathcal S}|$ denotes
the cardinality of the set for any finite set ${\mathcal S}$; ${\mathbb E}[\cdot]$ denotes the expectation
operator; ${\mathbb P}[\cdot]$ denotes the probability of an event; $[x]_i$ is the $i$-th element for any
vector $x$; $[P]_{ij}$ indicates the element in $i$-th row and $j$-th column for any matrix $P$;
if ${\bf z}$ is a discrete random variable which has $n$ values
and $\mu \in {\mathbb R}^n$ is a stochastic vector, then ${\bf z}
\sim \mu$ stands for ${\mathbb P}[{\bf z} = i] = [\mu]_i$ for all
$i \in \{1,\ldots,n \}$; ${\bf 1}_n \in {\mathbb R}^n$ denotes an
$n$-dimensional vector with all entries equal to one; ${\rm
dist}({\mathcal S},x)$ denotes the standard Euclidean distance of a vector $x$
from a set ${\mathcal S}$, i.e., ${\rm dist}({\mathcal S},x):=\inf_{y\in
{\mathcal S}} \|x-y\|_2$; for any ${\mathcal S} \subset {\mathbb R}^n$,
${\rm diam}({\mathcal S}):=\sup_{x \in {\mathcal S},y \in {\mathcal
S}}\|x-y\|_2 $ is the diameter of the set $\mathcal S$; for a convex
closed set $\mathcal S$, $\Gamma_{\mathcal S}(x)$ is the projection of $x$
onto the set $\mathcal S$, i.e., $\Gamma_{\mathcal S}(x):=\argmin_{y\in
{\mathcal S}} \|x-y\|_2$; a continuously differentiable function $f:{\mathbb R}^n \to {\mathbb R}$ is convex if $f(y)\ge f(x)+(y-x)^T \nabla f(x),\forall x,y \in {\mathbb R}^n$ and $\rho$-strongly convex if $f(y)\ge f(x)+(y-x)^T \nabla f(x)+(\rho/2) \|x-y\|^2,\forall x,y \in {\mathbb R}^n$~\cite[pp.~691]{bertsekas1999nonlinear}; $f_x(\bar x)$ is a subgradient of a convex
function $f:{\mathbb R}^n \to {\mathbb R}$ at a given vector $\bar
x \in {\mathbb R}^n$ when the following relation holds: $f(\bar x)
+ f_x (\bar x)^T (x-\bar x)\le f(x)$for all $x \in {\mathbb
R}^n$~\cite[pp.~209]{nedic2009subgradient}.

\subsection{Graph theory}
An undirected graph with the node set ${\mathcal V}$ and the edge set
${\mathcal E}\subseteq {\mathcal V}\times {\mathcal V}$ is denoted by ${\mathcal
G}=({\mathcal E},{\mathcal V})$. We define the neighbor set of node $i$ as
${\mathcal N}_i :=\{j\in {\mathcal V}:(i,j)\in {\mathcal E}\}$. The adjacency
matrix of $\mathcal G$ is defined as a matrix $W$ with $[W]_{ij} = 1$,
if and only if $(i,j) \in {\mathcal E}$. If $\mathcal G$ is undirected,
then $W=W^T$. A graph is connected, if there is a path between any
pair of vertices. The graph Laplacian is $L=H - W$, where $H$ is a
diagonal matrix with $[H]_{ii} = |{\mathcal N}_i|$. If the graph is
undirected, then $L$ is symmetric positive semi-definite. It holds
that $L {\bf 1}_{|{\mathcal V}|}=0$. If ${\mathcal G}$ is connected, $0$
is a simple eigenvalue of $L$, i.e., ${\bf 1}_{|{\mathcal V}|}$ is the unique eigenvector corresponding to $0$, and the span of ${\bf 1}_{|{\mathcal V}|}$ is the null space of $L$.

\subsection{Random communication network}
We will consider a random communication network model considered
in~\cite{lobel2011distributed}. In this paper, agents communicate
with neighboring agents and update their estimates at discrete
time instances $k \in \{0,1,\ldots\}$ over random time-varying
network ${\mathcal G}(k): = ({\mathcal E}(k),{\mathcal V}(k)),k \in \{ 1,2,
\ldots \} $. Let ${\mathcal N}_i (k): = \{ j \in {\mathcal V}(k):(i,j)\in
{\mathcal E}(k)\}$ be the neighbor set of agent $i$, $W(k)$ be the
adjacency matrix of ${\mathcal G}(k)$, and $H(k)$ be a diagonal matrix
with $[H(k)]_{ii}= |{\mathcal N}_i (k)|$. Then, the graph Laplacian of
${\mathcal G}(k)$ is $L(k):= H(k)-W(k)$. We assume that ${\mathcal
G}(k)$ is a random graph that is independent and identically
distributed over time $k$. A formal definition of the random
graph is given below.
\begin{assumption}
Let ${\mathcal F}:=(\Omega,{\mathcal B},\mu )$ be a probability space such
that $\Omega$ is the set of all $|{\mathcal V}|\times |{\mathcal V}|$
adjacency matrices, ${\mathcal B}$ is the Borel $\sigma$-algebra on
$\Omega$ and $\mu$ is a probability measure on ${\mathcal B}$. We
assume that for all $k\geq 0$, the matrix $W(k)$ is drawn from
probability space ${\mathcal F}$.
\end{assumption}
Define the expected value of the random matrices $W(k),H(k),L(k)$,
respectively, by
\begin{align*}
&{\bf W}:={\mathbb E}[W(k)],\quad {\bf H}:={\mathbb
E}[H(k)],\\
&{\bf L}:= {\mathbb E}[L(k)]={\bf H}-{\bf W},
\end{align*}
for all $k \ge 0$. An edge set induced by the positive elements of
the matrix ${\bf W}$ is ${\bf E}:=\{(j,i)\in {\mathcal V} \times {\mathcal
V}:[{\bf{W}}]_{ij}> 0\} $. Consider the corresponding graph $({\bf
E},{\mathcal V})$, which we refer to as the {\rm mean connectivity
graph}~\cite{lobel2011distributed}. We consider the following connectivity
assumption for the graph.
\begin{assumption}[Mean connectivity]\label{assumption:connected}
The mean connectivity graph $({\bf E},{\mathcal V})$ is connected.
\end{assumption}
Under~\cref{assumption:connected}, $0$ is a simple eigenvalue of
${\bf L}$~\cite[Lemma~1]{olfati2006flocking}. It implies that ${\bf L} {\bf 1}_{|{\mathcal V}|}=0$ holds, and later this assumption is used for the consensus of learning parameters.

\subsection{Reinforcement learning overview}\label{sec:RL-overview}
We briefly review a basic single-agent RL algorithm
from~\cite{sutton1998reinforcement} with linear function
approximation. A Markov
decision process (MDP) is characterized by a quadruple ${\mathcal M}: =
({\mathcal S},{\mathcal A},P,r,\gamma)$, where ${\mathcal S}$ is a finite
state space (observations in general), $\mathcal A$ is a finite action
space, $P(s,a,s'):={\mathbb P}[s'|s,a]$ represents the (unknown)
state transition probability from state $s$ to $s'$ given action
$a$, $\hat r:{\mathcal S}\times {\mathcal A}\times {\mathcal S}\to
[0,\sigma]$, where $\sigma >0$ is the bounded random reward
function, and $\gamma \in (0,1)$ is the discount factor. If action
$a$ is selected with the current state $s$, then the state
transits to $s'$ with probability $P(s,a,s')$ and incurs a random
reward $\hat r(s,a,s') \in [0,\sigma ]$ with expectation
$r(s,a,s')$. The stochastic policy is a map $\pi:{\mathcal S} \times
{\mathcal A}\to [0,1]$ representing the probability $\pi(s,a)={\mathbb
P}[a|s]$, $P^\pi$ denotes the transition matrix whose $(s,s')$
entry is ${\mathbb P}[s'|s] = \sum_{a \in {\mathcal A}} {P(s,a,s')\pi
(s,a)}$, and $d:{\mathcal S} \to {\mathbb R}$ denotes the stationary
distribution of the state $s\in {\mathcal S}$ under the behavior policy $\beta$. We also define
$r^\pi(s)$ as the expected reward given the policy $\pi$ and the
current state $s$, i.e.
\begin{align*}
r^\pi(s)&:= \sum_{a\in {\mathcal A}} {\sum_{s'\in {\mathcal S}} {\pi
(s,a)P(s,a,s')r(s,a,s')} }.
\end{align*}

The infinite-horizon discounted value function with policy $\pi$
and reward $\hat r$ is
\begin{align*}
&J^\pi(s):={\mathbb E} \left[ \left. \sum_{k = 0}^\infty {\gamma
^k \hat r(s_k,a_k,s_{k+1})} \right|s_0=s \right],
\end{align*}
where ${\mathbb E}$ stands for the expectation
taken with respect to the state-action trajectories following the
state transition $P^\pi$. Given pre-selected basis (or feature)
functions $\phi_1,\ldots,\phi_q:{\mathcal S}\to {\mathbb R}$, $\Phi
\in {\mathbb R}^{|{\mathcal S}| \times q}$ is defined as a full column
rank matrix whose $s$-th row vector is $\phi(s):=\begin{bmatrix} \phi_1(s)
&\cdots & \phi_q(s) \end{bmatrix}$. The goal of RL with the linear
function approximation is to find the weight vector $w$ such that
$J_{w}:=\Phi w$ approximates the true value function $J^{\pi}$.
This is typically done by minimizing the {\em mean-square Bellman
error} loss function~\cite{sutton2009fast}
\begin{align}
&\min_{w \in {\mathbb R}^q} {\rm MSBE}(w):=\frac{1}{2} \|
r^{\pi}+\gamma P^\pi\Phi w-\Phi w
\|_{D}^2,\label{eq:loss-function1}
\end{align}
where $D$ is a symmetric positive-definite matrix and $r^\pi \in {\mathbb R}^{|{\mathcal S}|}$ is a vector enumerating all $r^\pi(s), s\in {\mathcal S}$. For online
learning, we assume that $D$ is a diagonal matrix with positive
diagonal elements $d(s),s\in {\mathcal S}$. In the model-free learning, a stochastic gradient descent method can be applied with a stochastic estimates of the gradient $\nabla_w {\rm MSBE}(w)=(\gamma P^{\pi} \Phi-
\Phi)^TD(r^{\pi}+\gamma P^{\pi}\Phi w-\Phi w)$. The temporal difference
(TD)~learning~\cite{sutton1998reinforcement,bertsekas1996neuro}
with a linear function approximation is a stochastic gradient descent method with stochastic estimates of the approximate gradient
$\nabla_w {\rm MSBE}(w)\cong (-\Phi)^T D(r^{\pi}+\gamma
P^{\pi}\Phi w-\Phi w)$, which is obtained by dropping $\gamma P^{\pi} \Phi$ in $\nabla_w {\rm MSBE}(w)$. If the linear function approximation is used, then this algorithm
converges to an optimal solution of~\eqref{eq:loss-function1}. The GTD
in~\cite{sutton2009fast} solves instead the minimization of the
{\em mean-square projected Bellman error} loss function
\begin{align}
&\min_{w\in {\mathbb R}^q} {\rm MSPBE}(w):= \frac{1}{2}\|
\Pi (r^{\pi} + \gamma P^{\pi} \Phi w)-\Phi w \|_D^2,\label{eq:GDD(0)-loss}
\end{align}
where $\Pi$ is the projection onto the range space of $\Phi$,
denoted by $R(\Phi)$: $\Pi(x):=\argmin_{x'\in R(\Phi)}
\|x-x'\|_D^2$. The projection can be performed by the matrix
multiplication: we write $\Pi(x):=\Pi x$, where $\Pi:=\Phi(\Phi^T
D\Phi)^{-1}\Phi^T D$. Compared to the standard TD~learning, the
main advantage of the GTD algorithms~\cite{sutton2009convergent,sutton2009fast} is their
off-policy learning abilities.
\begin{remark}
Although its direct application to real problems is limited, the policy evaluation problem is a fundamental problem which is a critical building block to develop more practical policy optimization algorithms such as SALSA~\cite{rummery1994line} and actor-critic~\cite{konda2000actor} algorithms.
\end{remark}

Note that $d$ depends on the behavior policy, $\beta$, while $P^{\pi}$ and $r^{\pi}$ depend on the target policy, $\pi$, that we want to evaluate. This corresponds to the off-policy learning. The main problem is to obtain samples, $(s,a,\hat r,s')$ under $\pi$, from the samples under $\beta$. It can be done by the importance sampling or sub-sampling techniques~\cite{sutton2009convergent}. Throughout the paper, we mostly consider the case $\beta = \pi$ (on-policy) for simplicity. However, it can be generalized to the off-policy learning with simple modifications.

\section{Distributed reinforcement learning overview}\label{section:DRL-overview}
 In this section, we introduce the notion of the
distributed RL, which will be studied throughout the paper.
Consider $N$ RL agents labelled by $i \in \{ 1,\ldots,N\}=:{\mathcal
V}$. A multi-agent Markov decision process is characterized by
$({\mathcal S},\{ {\mathcal A}_i\}_{i\in {\mathcal
V}},P,\{\hat r_i\}_{i \in {\mathcal V}},\gamma)$, where $\gamma \in
(0,1)$ is the discount factor, ${\mathcal S}$ is a finite state
space, ${\mathcal A}_i$ is a finite action space of agent
$i$, $a:=(a_1,\ldots,a_N)$ is the joint
action, ${\mathcal A}:=\prod_{i=1}^N {{\mathcal
A}_i}$ is the
corresponding joint action space, $\hat
r_i:{\mathcal S} \times {\mathcal A} \times {\mathcal S} \to [0,\sigma]$,
$\sigma >0$, is a bounded random reward of agent $i$ with
expectation $r_i(s,a,s')$, and $P(s,a,s'):={\mathbb P}[s'|s,a]$
represents the transition model of the state $s$ with the
joint action $a$ and the corresponding joint action space ${\mathcal
A}$. The stochastic policy of agent $i$ is a mapping $\pi_i:{\mathcal
S} \times {\mathcal A}_i\to [0,1]$ representing the probability
$\pi_i(s,a_i)={\mathbb P}[a_i|s]$ and the corresponding joint
policy is $\pi(s,a):=\prod_{i=1}^N {\pi_i(s,a_i)}$. $P^{\pi}$
denotes the transition matrix, whose $(s,s')$ entry is ${\mathbb
P}[s'|s]=\sum_{a \in {\mathcal A}} {P(s,a,s')\pi (s,a)}$, $d:{\mathcal S}
\to {\mathbb R}$ denotes the stationary state distribution under
the policy $\pi$. In particular, if the joint action $a$ is
selected with the current state $s$, then the state transits to
$s'$ with probability $P(s,a,s')$, and each agent $i$ observes a
random reward $\hat r_i(s,a,s') \in [0,\sigma ]$ with expectation
$r_i(s,a,s')$. We assume that each agent does not have access to
other agents' rewards. For instance, there exists no centralized
coordinator; thereby each agent does not know other agents'
rewards. In another example, each agent/coordinator may not want
to uncover his/her own goal or the global goal for
security/privacy reasons. We denote by $r_i^\pi(s)$  the expected
reward of agent $i$, given the current state $s$
\begin{align*}
&r_i^\pi(s):= \sum_{a\in A} {\sum_{s'\in {\mathcal S}} {\pi(s,a)P(s,a,s')r_i(s,a,s')}}.
\end{align*}

Throughout the paper, a vector enumerating all $r_i^\pi(s), s\in {\mathcal S}$ is denoted by $r_i^\pi \in {\mathbb R}^{|{\mathcal S}|}$. In addition, denote by $P_i(s,a,s_i')$ the state transition
probability of agent $i$ given joint state $s$ and joint action
$a$. We can consider one of the following two scenarios throughout
the paper.
\begin{enumerate}

\item All agents can observe the identical state $s$. For
example, transitions of multiple ground robots avoiding collisions
with each other may depend on other robots actions and states, while they needs to know the global state, e.g., locations of all robots.

\item All agents observe different states, while each agent's state transition is
independent of the other agents' states and actions, i.e., they are fully decoupled. For example, each agent observes its own state which is sampled independently from the state transition probability of the MDP. For another instance, multiple robots navigating separated regions do not affect other agents' transitions.
\end{enumerate}

In this paper, we assume that the MDP with given $\pi$
has a stationary distribution.
\begin{assumption}
With a fixed policy $\pi$, the Markov chain $P^{\pi}$ is ergodic
with the stationary distribution $d$ with $d(s)>0, s\in {\mathcal S}$.
\end{assumption}

In addition, we summarize definitions and notations for some
important quantities below.
\begin{enumerate}
\item $D$ is defined as a diagonal matrix with diagonal entries
equal to those of $d$.

\item $J^\pi$ is the infinite-horizon discounted value function
with policy $\pi$ and reward $\hat r =(\hat r_1 +\cdots+\hat r_N
)/N$ defined as $J^\pi$ satisfying $J^\pi=\frac{1}{N}\sum_{i =
1}^N {r_i^{\pi}}+\gamma P^\pi J^\pi$.

\item We denote $\xi :=\min_{s \in {\mathcal S}} d(s)$.
\end{enumerate}

The goal is to learn an approximate value of the centralized
reward $\hat r=(\hat r_1+\cdots+\hat r_N)/N$ as stated below.
\begin{problem}[Multi-agent RL problem (MARLP)]\label{problem:multi-agent-RL}
In the multi-agent RL problem, the goal of each agent $i$ is to learn an approximate value
function of the centralized reward $\hat r=(\hat r_1+\cdots+ \hat r_N)/N$.
\end{problem}

Our first step to develop a decentralized RL algorithm to
solve~\cref{problem:multi-agent-RL} is to convert the problem into
an equivalent optimization problem. In particular, we can prove
that solving~\cref{problem:multi-agent-RL} is equivalent to
solving the optimization problem
\begin{align}
&\min_{w\in C}\sum_{i=1}^N {{\rm
MSPBE}_i(w)},\label{eq:distributed-opt0}
\end{align}
where ${\rm MSPBE}_i$ is defined as ${\rm MSPBE}_i(w):= \frac{1}{2}\|
\Pi (r_i^{\pi}+\gamma P^{\pi} \Phi w)-\Phi w
\|_D^2$ for all $i \in \{1,2,\ldots,N \}$, $C \subset {\mathbb R}^q$ is assumed to be a compact convex
set which includes an unconstrained global minimum
of~\eqref{eq:distributed-opt0}.
\begin{proposition}\label{prop:equivalance}
Solving~\eqref{eq:distributed-opt0} is equivalent to finding the unique solution $w^*$ to the projected Bellman
equation
\begin{align}
&\Pi\left( \frac{1}{N}\sum_{i=1}^N {r_i^{\pi}}+\gamma P^\pi \Phi
w^*\right)=\Phi w^*.\label{eq:projected-Bellman-eq}
\end{align}

Moreover, the solution is given by
\begin{align}
&w^*=(\Phi^T D(I-\gamma P^\pi)\Phi)^{-1}\Phi^T
D\frac{1}{N}\sum_{i=1}^N {r_i^\pi}.\label{eq:solution}
\end{align}
\end{proposition}
\begin{proof}
Since~\eqref{eq:distributed-opt0} is convex, $w^*$ is an
unconstrained global solutions, if and only if
\begin{align*}
&\nabla_w \sum_{i=1}^N {{\rm MSPBE}_i(w^*)}=0\\
\Leftrightarrow& -(\Phi^T D(I-\gamma P^{\pi})\Phi)^T (\Phi^T
D\Phi)^{-1}\Phi^T D \\
&\times\sum_{i=1}^N {(r_i^\pi-(I-\gamma P^{\pi})\Phi w^*)}= 0.
\end{align*}
Since $\Phi^T D(I-\gamma P^{\pi})\Phi$ is
nonsingular~\cite[pp.~300]{bertsekas1996neuro}, this implies
$(\Phi^T D\Phi)^{-1}\Phi^T D\sum_{i=1}^N{(r_i^{\pi}-(I-\gamma
P^{\pi})\Phi w^*)}=0$. Pre-multiplying the equation by $\Phi$ yields the projected
Bellman equation~\eqref{eq:projected-Bellman-eq}. A solution $w^*$
of the projected Bellman equation~\eqref{eq:projected-Bellman-eq}
exists~\cite[pp.~355]{bertsekas1996neuro}. To prove the second
statement, pre-multiply~\eqref{eq:projected-Bellman-eq} by $\Phi^T
D$ to have
\begin{align*}
&\Phi ^T D\left(\frac{1}{N}\sum_{i=1}^N {r_i^\pi}+\gamma
P^\pi \Phi w^*\right) = \Phi^T D\Phi w^*,
\end{align*}
where we use $\Pi:=\Phi (\Phi^T D\Phi)^{-1}\Phi^T D$ and $\Phi^T D
\Pi=\Phi^T D \Phi(\Phi^T D\Phi)^{-1}\Phi^T D=\Phi^T D$.
Rearranging terms, we have $\Phi ^T D\frac{1}{N}\sum_{i=1}^N {r_i^\pi}= \Phi^T D(I-\gamma
P^{\pi})\Phi w^*$. Since $\Phi^T D(I-\gamma P^{\pi})\Phi$ is
nonsingular~\cite[pp.~300]{bertsekas1996neuro}, pre-multiply both
sides of the above equation by $(\Phi^T D(I-\gamma
P^{\bar\pi})\Phi )^{-1}$ to obtain~\eqref{eq:solution}. The solution is unique because the objective function in~\eqref{eq:distributed-opt0} is strongly convex.
\end{proof}
\begin{remark}
If $\Phi=I_{|{\mathcal S}|}$, then the results are reduced to those
of the tabular representations. Therefore, all the developments in
this paper include both the tabular representation and the linear
function approximation cases.
\end{remark}

To develop a distributed algorithm, we first
convert~\eqref{eq:distributed-opt0} into the equivalent distributed
optimization problem~\cite{nedic2010constrained}\\
{\bf Distributed optimization form of MARLP:}
\begin{align}
&\min_{w_i\in C} \sum_{i=1}^N {{\rm
MSPBE}_i(w_i)}\label{eq:distributed-opt}\\
&{\rm subject\,\,to}\quad
w_1=w_2=\cdots=w_N,\label{eq:consensus-constraint}
\end{align}
where~\eqref{eq:consensus-constraint} implies the consensus among
$N$ copies of the parameter $w$. To make the problem more
feasible, we assume that the learning parameters $w_i$, $i\in
{\mathcal V}$, are exchanged through a random communication network
represented by the undirected graph ${\mathcal G}(k)=({\mathcal
E}(k),{\mathcal V}(k))$. In the next section, we will make several
conversions of~\eqref{eq:distributed-opt} to arrive at an
optimization form, which can be solved using a primal-dual saddle-point algorithm~\cite{nedic2009subgradient,nemirovski2009robust}.

\section{Stochastic primal-dual algorithm for saddle-point problem}\label{sec:primal-dual-algorithm}

The proposed RL algorithm is based on a saddle-point problem formulation of the distributed optimization problem~\eqref{eq:distributed-opt}. In this section, we briefly introduce the definition of the
saddle-point problem and a stochastic primal-dual algorithm~\cite{nemirovski2009robust} to find its solution.
\begin{definition}[Saddle-point~\cite{nedic2009subgradient}]\label{def:saddle-point}
Consider the map ${\mathcal L}:{\mathcal X} \times {\mathcal W} \to {\mathbb
R}$, where ${\mathcal X}$ and ${\mathcal W}$ are compact convex sets.
Assume that ${\mathcal L}(\cdot,w)$ is convex over ${\mathcal X}$ for all
$w \in {\mathcal W}$ and ${\mathcal L}(x,\cdot)$ is concave over ${\mathcal
W}$ for all $x \in {\mathcal X}$. Then, there exists a pair $(x^*,w^*
)$ that satisfies
\begin{align*}
&{\mathcal L}(x^*,w) \le {\mathcal L}(x^*,w^*)\le {\mathcal L}(x,w^*),\quad \forall (x,w) \in {\mathcal X} \times {\mathcal W}.
\end{align*}
The pair $(x^*,w^*)$ is called a saddle-point of ${\mathcal L}$. The saddle-point problem is defined as the problem of finding saddle
points $(x^*,w^*)$. It can be also defined as solving $\min_{x\in
{\mathcal X}} \max_{w\in {\mathcal W}} {\mathcal L}(x,w) = \max_{w \in {\mathcal
W}}\min_{x \in {\mathcal X}} {\mathcal L}(x,w)$.
\end{definition}

In our analysis, it will use the notion of approximate saddle-points in a geometric manner. In particular, the concept of the $\varepsilon$-saddle set is defined below.
\begin{definition}[$\varepsilon$-saddle set]\label{def:e-saddle-point}
For any $\varepsilon\geq 0$, the $\varepsilon$-saddle set is defined as
\begin{align*}
{\mathcal H}_\varepsilon :=&\{(x^*,w^*)\in {\mathcal X} \times {\mathcal W}:\\
& {\mathcal L}(x^*,w)-{\mathcal L}(x,w^*)\le \varepsilon,\forall x \in {\mathcal X},w \in {\mathcal W}\}.
\end{align*}
\end{definition}

From the definition, it is clear that ${\mathcal H}_0$ is the set of all saddle-points. The goal of the saddle-point problem is to find a saddle-point $(x^*,w^*)$ defined
in~\cref{def:saddle-point} over the set ${\mathcal X} \times {\mathcal
W}$. The stochastic primal-dual saddle-point algorithm in~\cite{nemirovski2009robust} can find a saddle-point when we have access to stochastic gradient estimates of function ${\mathcal L}$. It executes the following updates:
\begin{align}
&x_{k+1}=\Gamma_{\mathcal X} (x_k-\alpha_k ({\mathcal L}_x (x_k,w_k)+
\varepsilon_k)),\label{eq:promai-dual-subgrad1}\\
&w_{k+1}=\Gamma_{\mathcal W} (w_k+\alpha_k ({\mathcal L}_w
(x_k,w_k)+\xi_k)),\label{eq:promai-dual-subgrad2}
\end{align}
where ${\mathcal L}_x (x,w)$ and ${\mathcal L}_w(x,w)$ are the gradients of
${\mathcal L}(x,w)$ with respect to $x$ and $w$, respectively, and
$\varepsilon_k,\xi_k$ are i.i.d. random variables with zero means. To
proceed, define the history of the algorithm until time $k$, ${\mathcal F}_k :=(\varepsilon_0,\ldots ,\varepsilon _{k-1},\xi_0,\ldots,\xi_{k-1},x_0,\ldots,x_k,w_0,\ldots,w_k )$ related to~\cref{algo:DGTD}. In the following result, we provide a finite-time convergence of the primal-dual algorithm with high probabilities.
\begin{proposition}\label{prop:convergence1}
Assume that there exists a constant $C >0 $ such that
\begin{align}
&\| {\mathcal L}_x(x_k,w_k)+\varepsilon_k
\|_2\le C,\label{eq:1}\\
&\|{\mathcal L}_w(x_k,w_k )+\xi_k\|_2 \le C,\label{eq:2}\\
&{\rm diam}({\mathcal X}) \le C,\quad {\rm diam}({\mathcal W})\le C.\label{eq:3}
\end{align}
In addition, we assume that the step-size sequence $(\alpha_k)_{k=0}^\infty$
satisfies $\alpha_k =\alpha_0 /\sqrt{k+1}$. Let $\hat x_T =
\frac{1}{T}\sum_{k=0}^{T-1} {x_k}$ and $\hat w_T =
\frac{1}{T}\sum_{k=0}^{T-1} {w_k}$ be the averaged dual iterates
generated by~\eqref{eq:promai-dual-subgrad1}
and~\eqref{eq:promai-dual-subgrad2} with $T \ge 1$. Then, for any
$\varepsilon> 0,\delta\in (0,1)$, if $T \ge\max \{\Omega_1,\Omega_2\}$, then
\begin{align*}
&{\mathbb P}[(\hat x_T,\hat w_T ) \in H_\varepsilon  ] \ge 1 - \delta,
\end{align*}
where
\begin{align*}
\Omega_1:=&\frac{8C^2((\alpha_0+2)^2 C^2+(\alpha_0+4)\varepsilon/6)}{\varepsilon ^2} \ln\left(\frac{1}{\delta}\right),\\
\Omega_2:=&\frac{4C^4 (2\alpha_0^{-1}+\alpha_0)^2}{\varepsilon^2}.
\end{align*}
\end{proposition}
\begin{remark}
Convergence of the stochastic primal-dual algorithm was proved in~\cite[Section~3.1]{nemirovski2009robust}. Compared to the analysis in~\cite{nemirovski2009robust}, the analysis in~\cref{prop:convergence1} poses some refined aspects tailored to our purposes. First, the analysis in~\cite[Section~3.1]{nemirovski2009robust} considers a solution which is so-called the sliding average of the primal and dual iterations, while the solution considered in~\cref{prop:convergence1} uses an average of the entire iteration until the current step, which is simpler.
\end{remark}

\section{Saddle-point formulation of MARLP}

In the previous section, we introduced the notion of the saddle-point and a stochastic primal-dual algorithm to find it. In this section, we study a saddle-point formulation of the distributed optimization~\eqref{eq:distributed-opt} as a next step. Once obtained, the MARLP can be solved by using the stochastic primal-dual algorithm. For notational simplicity, we first introduce stacked vector and matrix notations.
\begin{align*}
&\bar w: = \begin{bmatrix}
   w_1 \\
    \vdots \\
   w_N \\
\end{bmatrix},\quad \bar r^\pi  := \begin{bmatrix}
   r_1^\pi \\
    \vdots \\
   r_N^\pi \\
\end{bmatrix},\quad \hat r(s,a,s') = \begin{bmatrix}
   \hat r_1 (s,a,s') \\
    \vdots   \\
   \hat r_N (s,a,s') \\
\end{bmatrix},\\
&\bar P^{\pi}:=I_N \otimes P^{\pi},\quad \bar {\bf L}:={\bf L}
\otimes I_{|{\mathcal S}|},\quad \bar D:=I_N \otimes D,\\
& \bar\Phi:=I_N \otimes \Phi,\quad \bar \Pi := I_N \otimes \Pi,\quad \bar B :=
\bar\Phi^T \bar D(I_{N|{\mathcal S}|}-\gamma\bar P^{\pi})\bar \Phi.
\end{align*}

Using those notations, the {\rm MSPBE} loss function
in~\eqref{eq:distributed-opt} can
be compactly expressed as
\begin{align*}
&\sum_{i=1}^N {\rm MSPBE}_i(w_i)\\
=& \frac{1}{2}(\bar\Phi^T \bar D\bar r^\pi-\bar B\bar w)^T (\bar
\Phi^T \bar D\bar \Phi)^{-1} (\bar\Phi^T \bar D\bar r^\pi-\bar
B\bar w),
\end{align*}
where $\otimes$ is the Kronecker's product. Note that by the mean
connectivity~\cref{assumption:connected}, the consensus
constraint~\eqref{eq:consensus-constraint} can be expressed as
$\bar {\bf L}\bar w=0$, as ${\bf L}$ has a simple eigenvalue $0$
with its corresponding eigenvector ${\bf 1}_{|{\mathcal
S}|}$~\cite[Lemma~1]{olfati2006flocking}. Motivated by the
continuous-time consensus optimization algorithms
in~\cite{wang2010control,wang2011control,gharesifard2014distributed},
we convert the problem~\eqref{eq:distributed-opt} into the
augmented Lagrangian
problem~\cite[sec.~4.2]{bertsekas1999nonlinear}
\begin{align}
&\mathop{\min}_{\bar w} \frac{1}{2}(\bar\Phi^T \bar D\bar r^\pi-\bar B\bar w)^T (\bar\Phi^T \bar D\bar\Phi)^{-1}(\bar\Phi^T \bar D\bar r^\pi-\bar B\bar w)\nonumber\\
+&\bar w^T \bar {\bf L} \bar {\bf L} \bar w\label{eq:optimization1}\\
&{\rm subject\,\,to}\quad \bar {\bf L}\bar w = 0,\nonumber
\end{align}
where a quadratic penalty term $\bar w^T \bar {\bf L} \bar {\bf L}
\bar w$ for the equality constraint $\bar {\bf L}\bar w = 0$ is
introduced. If the model is known, the above problem is an
equality constrained quadratic programming problem, which can be
solved by means of convex optimization methods~\cite{Boyd2004}.
Otherwise, the problem can be still solved  using stochastic
algorithms with observations. The latter case is our main concern.
To develop model-free stochastic algorithms, some issues need to
be taken into account. First, to estimate a stochastic estimate of the gradient, we need to assume that at least two independent next state samples can be drawn from any current state, which is impossible in most practical applications. The problem is often called the double sampling problem~\cite{bertsekas1996neuro}. Second, the inverse matrix $(\bar\Phi^T \bar D\bar\Phi)^{-1}$ in the
objective function~\eqref{eq:optimization1} needs to be removed. In particular, the main reason we use the linear function approximation is due to the large size of the state-space to the extent that enumerating numbers in the value vector is computationally demanding or even not possible. The computation of the inverse $(\bar\Phi^T \bar D\bar\Phi)^{-1}$ is not possible due to both its computational complexity and the existence of the matrix $\bar D$ including the stationary state distribution, which is assumed to be unknown in most RL settings. In GTD~\cite{sutton2009fast}, this problem is resolved  using a
dual problem~\cite{macua2015distributed}. Following the same
direction, we convert~\eqref{eq:optimization1} into the equivalent
optimization problem
\begin{align}
&\mathop{\min}_{\bar\varepsilon,\bar h,\bar w} \frac{1}{2}\bar
\varepsilon^T (\bar\Phi^T\bar D\bar\Phi)^{-1}\bar \varepsilon+\frac{1}{2}\bar h^T \bar h \label{eq:optimization2}\\
&{\rm subject\,\,to}\quad \begin{bmatrix}
   \bar B & I & 0\\
   \bar {\bf L} & 0 & -I\\
   \bar {\bf L} & 0 & 0\\
\end{bmatrix} \begin{bmatrix}
   \bar w\\
   \bar\varepsilon\\
   \bar h\\
\end{bmatrix}+ \begin{bmatrix}
   -\bar\Phi^T \bar D\bar r^{\bar\pi}\\
   0\\
   0\\
\end{bmatrix}=0,\nonumber
\end{align}
where $\bar\varepsilon$ and $\bar h$ are newly introduced
parameters. The next key step is to derive its Lagrangian dual
problem~\cite{Boyd2004}, which can be obtained  using standard
approaches~\cite{Boyd2004}.
\begin{proposition}\label{prop:dual-problem}
The Lagrangian dual problem of~\eqref{eq:optimization2} is given
by
\begin{align}
&\min_{\bar\theta,\bar v,\bar\mu}\psi (\bar\theta,\bar
v,\bar\mu)\label{eq:dual-problem1}\\
&{\rm subject\,\,to}\quad\bar B^T \bar\theta-\bar {\bf L}^T\bar
v- \bar {\bf L}^T\bar\mu=0,\nonumber
\end{align}
where $\psi(\bar\theta,\bar v,\bar\mu ):= \frac{1}{2}\bar\theta^T
(\bar \Phi^T \bar D\bar\Phi)\bar \theta-\bar\theta^T\bar\Phi^T\bar
D\bar r^{\bar\pi}+\frac{1}{2}\bar v^T \bar v$.
\end{proposition}
\begin{proof}
The dual problem can be obtained  using  standard
manipulations as in~\cite[Chap.~5]{Boyd2004}. Define the Lagrangian
function
\begin{align*}
&{\mathcal L}(\bar \varepsilon,\bar h,\bar w,\bar\theta ,\bar
v,\bar\mu)\\
=& \frac{1}{2}\bar \varepsilon^T (\bar\Phi^T \bar D \bar\Phi)^{-1} \bar\varepsilon+\frac{1}{2}\bar h^T \bar h + \bar\theta ^T (\bar\Phi^T \bar D\bar r^{\pi}-\bar B\bar w-\bar\varepsilon )\\
&+\bar v^T(\bar{\bf L}\bar w-\bar h) + \bar\mu^T \bar {\bf L}\bar w\\
=& \frac{1}{2}\bar\varepsilon^T (\bar\Phi^T \bar D\bar \Phi)^{
-1} \bar \varepsilon-\bar \theta ^T \bar \varepsilon+
\frac{1}{2}\bar h^T\bar h-\bar v^T \bar h+\bar \theta^T\bar
\Phi^T \bar D\bar r^{\pi}\\
&- (\bar\theta^T \bar B - \bar v^T \bar {\bf L} -\bar \mu ^T \bar
{\bf L})\bar w,
\end{align*}
where $\bar\theta,\bar v,\bar\mu$ are Lagrangian multipliers. If
we fix $(\bar \theta ,\bar v,\bar \mu)$, then the problem $\min
_{\bar\varepsilon,\bar h,\bar w} {\mathcal L}(\bar \varepsilon ,\bar
h,\bar w,\bar \theta ,\bar v,\bar \mu )$ has a finite optimal
value, when $\bar \theta ^T \bar B - \bar v^T \bar {\bf L} - \bar
\mu ^T \bar {\bf L} = 0$. The optimal solutions satisfy $\bar
\varepsilon=(\bar \Phi ^T \bar D\bar \Phi)\bar \theta ,\bar h =
\bar v$. Plugging them into the Lagrangian function, the dual
problem is obtained.
\end{proof}

One can observe that the inverse matrix $(\bar\Phi^T \bar
D\bar\Phi)^{-1}$ no more appears in the dual
problem~\eqref{eq:dual-problem1}. To
solve~\eqref{eq:dual-problem1}, we again construct the following
Lagrangian function of~\eqref{eq:dual-problem1} as
in~\cite{macua2015distributed}:
\begin{align}
&{\mathcal L}(\bar\theta,\bar v,\bar\mu,\bar w):=\psi(\bar\theta,\bar
v,\bar\mu)+[\bar B^T \bar\theta-\bar {\bf L}^T\bar v-\bar {\bf
L}^T\bar\mu]^T \bar w,\label{eq:Lagrangian1}
\end{align}
where $\bar w$ is the Lagrangian multiplier. We further modify~\eqref{eq:Lagrangian1} by adding the term $-(\kappa/2) \bar w^T \bar {\bf L} \bar w$:
\begin{align}
{\mathcal L}(\bar\theta,\bar v,\bar\mu,\bar w):=&\psi(\bar\theta,\bar
v,\bar\mu)+[\bar B^T \bar\theta-\bar {\bf L}^T\bar v-\bar {\bf
L}^T\bar\mu]^T \bar w\nonumber\\
& -(\kappa/2)\bar w^T \bar {\bf L} \bar w,\label{eq:Lagrangian2}
\end{align}
where $\kappa \geq 0$ is a design parameter. Note that the solution of the original problem is not changed for any $\kappa\geq 0$. The term, $-(\kappa/2)\bar w^T \bar {\bf L} \bar w$, is added to accelerate the convergence in terms of the consensus of $\bar w$.

Since the Lagrangian function~\eqref{eq:Lagrangian2} is convex-concave, the solutions of the optimization in~\eqref{eq:dual-problem1}
are identical to solutions $(\bar\theta^*,\bar v^*,\bar\mu^*,\bar
w)$ of the corresponding saddle-point problem~\cite{nedic2009subgradient}
\begin{align}
&\max_{\bar w} \min_{\bar\theta,\bar v,\bar \mu}
{\mathcal L}(\bar\theta,\bar v,\bar\mu,\bar
w)=\min_{\bar\theta,\bar v,\bar\mu} \max_{\bar w}
{\mathcal L}(\bar\theta,\bar v,\bar\mu,\bar w).\label{eq:saddle}
\end{align}
or equivalently,
\begin{align}
&{\mathcal L}(\bar\theta^*,\bar v^*,\bar \mu^*,\bar w)\le {\mathcal
L}(\bar\theta^*,\bar v^*,\bar\mu^*,\bar w^*) \le {\mathcal
L}(\bar\theta,\bar v,\bar\mu,\bar w^*),\label{eq:saddle2}
\end{align}
for all $(\bar\theta,\bar v,\bar\mu,\bar w)$. Now, the saddle-point problem in~\eqref{eq:saddle} can be solved by using the stochastic primal-dual algorithm~\cite{nemirovski2009robust}.

\section{Solution analysis}
In the previous section, we derived a saddle-point formulation of the distributed optimization~\eqref{eq:distributed-opt}. In this section, we rigorously analyze the set of saddle-points. In particular, we obtain an exact formulations of the set of saddle-points which solve~\eqref{eq:saddle}. The explicit formulations of the saddle-points will be used in subsequent sections to develop the proposed RL algorithm. According to the standard results in convex optimization~\cite[Section 5.5.3, pp. 243]{Boyd2004}, any saddle-point $(\bar\theta^*,\bar v^*,\bar\mu^*,\bar w^*)$
satisfying~\eqref{eq:saddle2} must satisfy the following KKT
condition although its converse is not true in general:
\begin{align}
0 =& \nabla_{\bar\theta} {\mathcal L}(\bar\theta^*,\bar
v^*,\bar\mu^*,\bar w^*)\nonumber\\
=&(\bar\Phi^T \bar D\bar\Phi)\bar\theta^*-\bar\Phi^T \bar D\bar
r^{\pi}+\bar\Phi^T \bar D(I_{N|{\mathcal S}|} - \gamma \bar P^{\pi})\bar \Phi \bar w^*,\nonumber\\
0=& \nabla_{\bar v}{\mathcal
L}(\bar\theta^*,\bar v^*,\bar\mu^*,\bar w^*)=\bar v^* - \bar {\bf L}\bar w^*,\nonumber\\
0=& \nabla_{\bar\mu} {\mathcal L}(\bar\theta^*,\bar v^*,\bar\mu^*,\bar
w^*)=\bar {\bf L}\bar w^*,\nonumber\\
0=& \nabla_{\bar w} {\mathcal L}(\bar\theta^*,\bar
v^*,\bar\mu^*,\bar w^*)\nonumber\\
=&\bar {\bf L}\bar v^* +\bar {\bf L}\bar \mu^* -\bar\Phi^T
(I_{N|{\mathcal S}|}-\gamma\bar P^{\pi})^T \bar
D\bar\Phi\bar\theta^* - \kappa\bar {\bf L}\bar w^*.\label{eq:KKT-points}
\end{align}
However, by investigating the KKT points, we can obtain useful information on the saddle-points. We first establish the fact that the set of KKT points corresponds to the set of optimal solutions of the consensus optimization problem~\eqref{eq:consensus-constraint}.
\begin{proposition}\label{prop:stationary-points}
The set of all the KKT points satisfying~\eqref{eq:KKT-points} is
given by
\begin{align*}
&{\mathcal R}: = \{ \bar\theta^* \}\times \{\bar v^*\}\times {\mathcal
F}^* \times \{{\bf 1}_N  \otimes w^* \},
\end{align*}
where $\bar v^*=0$, $w^*$ is given in~\eqref{eq:solution} (the unique solution of the
projected Bellman equation~\eqref{eq:projected-Bellman-eq}),
\begin{align*}
\bar\theta^*=& (\bar \Phi^T \bar D\bar\Phi)^{-1} \bar \Phi^T\bar D(-\bar r^\pi+\bar \Phi \bar w^* - \gamma\bar P^{\bar\pi} \bar\Phi\bar w^* )\\
 =& (\bar\Phi^T \bar D\bar \Phi)^{-1} \bar \Phi^T \bar D\left( -\bar r^\pi+{\bf 1}_N \otimes \frac{1}{N}\sum_{i=1}^N {r_i^{\pi}}
\right),
\end{align*}
and ${\mathcal F}^*$ is the set of all solutions to the linear
equation for $\bar \mu$
\begin{align}
&{\mathcal F}^*:= \{\bar\mu:\bar {\bf L}\bar \mu=\bar\Phi^T (I_{N|{\mathcal S}|} - \gamma\bar P^\pi)^T \bar D\bar
\Phi\bar\theta^*\}.\label{eq:set-F}
\end{align}
\end{proposition}
\begin{proof}

The KKT condition in~\eqref{eq:KKT-points} is equivalent to the
linear equations:
\begin{align}
&\nabla_{\bar\theta} {\mathcal L}(\bar\theta^*,\bar
v^*,\bar\mu^*,\bar w^*)\nonumber\\
=&(\bar\Phi^T \bar D\bar\Phi)\bar\theta^*-\bar\Phi^T \bar D\bar r^{\pi}+\bar\Phi^T \bar D(I_{N|{\mathcal S}|} - \gamma \bar P^{\pi})\bar \Phi \bar w^*\\
=& 0,\label{eq:appendix:eq1}\\
&\nabla_{\bar v}{\mathcal L}(\bar\theta^*,\bar v^*,\bar\mu^*,\bar w^*)=\bar v^* - \bar {\bf L}\bar w^* = 0,\label{eq:appendix:eq2}\\
&\nabla_{\bar\mu} {\mathcal L}(\bar\theta^*,\bar v^*,\bar\mu^*,\bar w^*)=\bar {\bf L}\bar w^* = 0,\label{eq:appendix:eq4}\\
&\nabla_{\bar w} {\mathcal L}(\bar\theta^*,\bar
v^*,\bar\mu^*,\bar w^*)\\
=&\bar {\bf L}\bar v^* +\bar {\bf L}\bar \mu^* -\bar\Phi^T (I_{N|{\mathcal S}|}-\gamma\bar P^{\pi})^T \bar D\bar\Phi\bar\theta^* - \kappa\bar {\bf L}\bar w^*\\
&=0.\label{eq:appendix:eq3}
\end{align}
Since the mean connectivity graph $({\bf E},{\mathcal V})$ of ${\mathcal
G}(k)$ is connected by~\cref{assumption:connected}, the dimension
of the null space of ${\bf L}$ is one. Therefore, ${\rm span}({\bf 1}_{|{\mathcal V}|})$ is the null space, and~\eqref{eq:appendix:eq4}
implies the consensus $w^*=w^*_1=\cdots=w^*_N$.
Plugging~\eqref{eq:appendix:eq4} into \eqref{eq:appendix:eq2}
yields $\bar v^*=0$. With ${\bar v}^*=0$,~\eqref{eq:appendix:eq3}
is simplified to
\begin{align}
&\bar {\bf L} \bar\mu^* = \bar\Phi^T(I_{N|{\mathcal S}|}-\gamma\bar P^{\pi})^T \bar D\bar\Phi \bar\theta^*.\label{eq:appendix:eq12}
\end{align}
In addition, from~\eqref{eq:appendix:eq1}, the stationary point
for $\bar \theta$ satisfies
\begin{align}
&\bar\theta^*=(\bar\Phi^T\bar D\bar\Phi )^{-1} \bar\Phi^T \bar D(\bar r^{\pi}-\bar\Phi\bar w^*+\gamma \bar P^{\pi}\bar\Phi\bar
w^*).\label{eq:appendix:eq13}
\end{align}

Plugging the above equation into~\eqref{eq:appendix:eq12} yields
\begin{align}
\bar {\bf L} \bar\mu^* =& \bar\Phi^T(I_{N|{\mathcal S}|}-\gamma\bar
P^{\bar\pi})^T \bar D\bar\Phi\bar\theta^*\nonumber\\
=&\bar\Phi^T (I_{N|{\mathcal S}|}-\gamma\bar P^{\bar\pi})^T \bar
D\bar\Phi(\bar\Phi^T\bar D\bar\Phi)^{-1}\bar\Phi^T \bar D\nonumber\\
&\times (\bar r^{\pi}-\bar\Phi\bar w^* + \gamma\bar
P^{\pi}\bar\Phi\bar w^*)\label{eq:appendix:eq7}.
\end{align}
Multiplying~\eqref{eq:appendix:eq7} by $({\bf 1}\otimes I)^T$ on the left results in
\begin{align*}
&(\Phi^T D(I_{|{\mathcal S}|}-\gamma P^{\pi}
)\Phi)^T (\Phi^T D\Phi )^{-1}\Phi^T D\\
&\times \left( \frac{1}{N}\sum_{i=1}^N {r_i^{\pi_i}}+\gamma
P^{\pi}\Phi w^*-\Phi w^*\right) = 0.
\end{align*}
Since $\Phi^T D(I-\gamma\bar P^\pi)\Phi$ is
nonsingular~\cite[pp.~300]{bertsekas1996neuro}, pre-multiplying
both sides of the last equation with $((\Phi^T D(I-\gamma\bar
P^\pi)\Phi )^T)^{-1}$ results in
\begin{align}
&(\Phi^T D\Phi)^{-1} \Phi^T D\left( \frac{1}{N}\sum_{i = 1}^N {r_i^{\pi_i}}  + \gamma P^\pi \Phi w^*-\Phi w^*
\right)=0.\label{eq:17}
\end{align}
Pre-multiplying~\eqref{eq:17} with $\Phi^T$ from  left yields the projected Bellman equation in~\cref{prop:equivalance}, and $w^*$
is any of its solutions. In particular,
multiplying~\eqref{eq:appendix:eq1} by $({\bf 1} \otimes I)^T$ from left, a KKT point for $\bar w^*$ is expressed as $\bar w^*={\bf 1} \otimes w^*$ with
\begin{align*}
w^*=& (\Phi^T D(I-\gamma P^{\pi} )\Phi)^{-1} \Phi^T D\\
&\times \left( \frac{1}{N}\sum_{i=1}^N {r_i^{\pi_i}} - \Pi \left(-\frac{1}{N}\sum_{i=1}^N {r_i^{\pi_i}} + \Phi w^ - \gamma P^\pi \Phi w^*\right)\right)\\
=& (\Phi^T D(I_{|{\mathcal S}|} - \gamma P^{\pi})\Phi)^{-1} \Phi^T
D\frac{1}{N}\sum_{i=1}^N {r_i^{\pi_i}}.
\end{align*}
From~\eqref{eq:appendix:eq7}, ${\bar\mu}^*$ is any solution of the linear equation~\eqref{eq:appendix:eq7}. Lastly,~\eqref{eq:17} can be rewritten as
\begin{align*}
&0 = (\bar \Phi^T \bar D\bar \Phi)^{-1} \bar \Phi ^T \bar D\left({\bf 1}_N \otimes \frac{1}{N}\sum_{i=1}^N {r_i^{\pi_i}}+\gamma \bar P^\pi  \Phi\bar w^* -\Phi \bar w^*\right).
\end{align*}
Subtracting~\eqref{eq:appendix:eq13} by the last term, we
obtain $\bar \theta^* =(\bar\Phi ^T \bar D\bar \Phi )^{-1} \bar \Phi ^T \bar D\left(-\bar r^\pi +{\bf 1}_N\otimes
\frac{1}{N}\sum_{i=1}^N {r_i^{\pi_i} }\right)$. This completes the
proof.
\end{proof}
Since the set of saddle-points of ${\mathcal L}$ in~\eqref{eq:Lagrangian1} is a subset of the KKT points, we can
estimate a potential structure of the set of saddle-points.
\begin{corollary}\label{corollary:saddle-point}
The set of all the saddle-points, ${\mathcal H}_0$, satisfying~\eqref{eq:saddle2} is
given by ${\mathcal H}_0:= \{ \bar\theta^* \}\times \{\bar v^*\}\times
\tilde {\mathcal F}^* \times \{ {\bf 1}_N  \otimes w^*\}$, where $\bar v^*=0$, $w^*$ is the unique solution of the
projected Bellman equation~\eqref{eq:projected-Bellman-eq},
$\tilde {\mathcal F}^*$ is some subset of ${\mathcal F}^*$, ${\mathcal F}^*$
and $\bar \theta ^*$ are defined in~\cref{prop:stationary-points}.
\end{corollary}

According to~\cref{corollary:saddle-point}, the set of KKT points corresponding to $\bar\theta$, $\bar v$, and $\bar w$ is a singleton $\{ \bar\theta^*\}\times \{\bar v^*\}\times \{ {\bf 1}_N  \otimes w^*\}$. Therefore, it is the unique saddle-point corresponding to $\bar\theta$, $\bar v$, and $\bar w$. On the other hand, $\tilde {\mathcal F}^*$ is a set. We can prove that $\tilde {\mathcal
F}^*$ is an affine space.
\begin{lemma}\label{lemma:affine-space-property}
$\tilde {\mathcal F}^*$ is an affine space.
\end{lemma}
\begin{proof}
By the saddle-point property in~\eqref{eq:saddle2}, $\hat {\mathcal
F}^* =\tilde {\mathcal F}^*$ if and only if $L(\bar\theta^*,\bar
v^*,\bar\mu ^*,\bar w^*)=L(\bar\theta^*,\bar v^*,\bar\mu,\bar
w^*),\forall \bar\mu\in \hat {\mathcal F}^*$, which is equivalent to
$\bar\mu^T \bar L\bar w^*= \bar\mu^{*T}\bar L\bar w^* ,\forall
\bar \mu \in \hat {\mathcal F}^*$, proving that $\hat {\mathcal F}^* =
\tilde {\mathcal F}^*$ is an affine space.
\end{proof}

By~\cref{lemma:affine-space-property}, we can obtain an explicit formulation of a point in $\tilde {\mathcal F}^*$.
\begin{proposition}\label{prop:saddle-point-mu}
We have $\bar\mu^*=\bar{\bf L}^\dag \bar \Phi^T (I - \gamma \bar
P^\pi )^T \bar D\bar \Phi \bar \theta^* \in \tilde {\mathcal F}^*$.
\end{proposition}
\begin{proof}
Since ${\mathcal F}^*$ is the set of solutions of the linear equation
$\bar {\bf L}\bar\mu=\bar \Phi ^T (I-\gamma \bar P^\pi)^T \bar
D\bar \Phi \bar \theta^*$, ${\mathcal F}^*$ is the set of general
solutions of the linear equation, which are given by the affine
space $\bar\mu = \bar {\bf L}^\dag \bar \Phi^T (I-\gamma \bar
P^\pi )^T \bar D\bar \Phi \bar \theta^*+(\bar{\bf L}^\dag \bar
{\bf L} - I)z$, where $\bar {\bf L}^\dag$ is a pseudo-inverse of
$\bar {\bf L}$ and $z \in {\mathbb R}^{|{\mathcal S}|N}$ is arbitrary.
In addition, since $\tilde {\mathcal F}^*\subseteq {\mathcal F}^*$ and
$\tilde {\mathcal F}^*$ is also affine
by~\cref{lemma:affine-space-property}, one concludes that $\bar
\mu = \bar {\bf L}^\dag \bar \Phi^T (I - \gamma\bar P^\pi)^T \bar
D\bar \Phi \bar \theta^*\in\tilde {\mathcal F}^*$.
\end{proof}
For some technical reasons that will become clear later,
algorithms to find a solution need to confine the search space of
an algorithm to compact and convex sets which include at least one
saddle-point in $\tilde {\mathcal R}$
in~\cref{corollary:saddle-point}. To this end, we compute a bound
on at least one saddle-point $(\bar\theta^*,\bar v^*,\bar
\mu^*,\bar w^*)$ in the following lemma.
\begin{lemma}\label{lemma:bound-lemma3}
$\bar w^*$, $\bar v^*$ and $\bar \theta ^*$ satisfy the following
bounds:
\begin{align*}
&\| \bar w^* \|_\infty\le \frac{1}{1 - \alpha}\sqrt {\frac{|{\mathcal
S}|}{\lambda_{\min } (\Phi^T \Phi)}} \left( \frac{1}{\sqrt \xi}\|
\Pi J^\pi-J^\pi \|_D
+ \sigma \right)\\
&\| \bar v^*\|_\infty \le c_{\bar v},\quad \forall
c_{\bar v}\geq 0,\\
&\| \bar\theta^* \|_\infty   \le 2\sigma |{\mathcal S}|\sqrt
{\frac{N}{\xi\lambda_{\min} (\Phi ^T \Phi)}},
\end{align*}
where $\xi :=\min_{s\in {\mathcal S}}d(s)$ as defined
in~\cref{section:DRL-overview}. Moreover, there exists a
$\bar\mu^*\in \tilde {\mathcal F}^*$ such that
\begin{align*}
&\|\bar\mu^* \|_\infty \le \|{\bf L}^\dag\|_\infty \|\Phi\|_\infty ^2 2\sigma |{\mathcal S}|^2 \sqrt {\frac{N}{\xi \lambda_{\min}(\Phi^T\Phi)}}.
\end{align*}
For the pseudo-inverse of the graph Laplacian
in~\cite{ghosh2008minimizing}, we can use the expression $ {\bf
L}^\dag = ({\bf L} + {\bf 1}_{N} {\bf 1}_{N}^T /N)^{-1}-{\bf
1}_{N} {\bf 1}_{N}^T /N$.
\end{lemma}
\begin{proof}
To prove~\cref{lemma:bound-lemma3}, we will first prove a bound on
$w^* \in {\mathcal W}^*$.\\
{\bf Claim:}
If $w^*$ is an optimal solution presented
in~\cref{prop:equivalance}, then
\begin{align*}
&\| w^*\|_\infty\le \frac{1}{1 - \alpha}\sqrt {\frac{|{\mathcal
S}|}{\lambda_{\min}(\Phi^T\Phi)}} \left( \frac{1}{\sqrt\xi
}\left\| \Pi J^\pi- J^\pi\right\|_D + \sigma \right).
\end{align*}
{\bf Proof of Claim:}
We first bound the term $\| \Phi w^* \|_\infty$ as follows:
\begin{align*}
\| \Phi w^* \|_\infty & = \| \Phi w^* - J^\pi
+ J^\pi \|_\infty\\
& \le \| \Phi w^*- J^\pi \|_\infty + \|
J^\pi \|_\infty\\
&\le \| \Phi w^*  - J^\pi \|_2  + \| J^\pi
\|_\infty\\
&\le \frac{1}{\xi } \| \Phi w^*  - J^\pi \|_D  +
\| J^\pi\|_\infty\\
&\le \frac{1}{\sqrt{\xi}}\frac{1}{1 - \alpha}\| \Pi
J^\pi-J^\pi \|_D  + \| J^\pi \|_\infty\\
&\le \frac{1}{\sqrt{\xi}}\frac{1}{1 - \alpha} \| \Pi J^\pi - J^\pi
\|_D + \frac{\sigma }{1-\alpha},
\end{align*}
where the first inequality follows from the triangle inequality,
the second inequality uses $\| \cdot \|_\infty \leq \| \cdot
\|_2$, the third inequality uses $\sqrt {x^T x}  \le \sqrt {x^T
Dx/\min_{s \in {\mathcal S}} d(s)}= \sqrt {x^T Dx} / {\sqrt \xi}$, the
fourth inequality comes from~\cite[Prop.~6.10]{bertsekas1996neuro},
and the last inequality uses the bound on the rewards. On the
other hand, its lower bound can be obtained as
\begin{align*}
\| \Phi w^* \|_\infty \ge&\frac{1}{{\sqrt {|{\mathcal S}|} }} \| \Phi
w^* \|_2\\
\ge& \frac{1}{\sqrt {|{\mathcal S}|}} \| w^*\|_2 \sqrt{\lambda_{\min}(\Phi^T \Phi )} \\
\ge& \frac{1}{{\sqrt {|{\mathcal S}|} }}\| w^* \|_\infty
\sqrt{\lambda_{\min }(\Phi ^T \Phi )},
\end{align*}
where the first inequality comes from $\| v \|_2  \le \sqrt{n} \|
v \|_\infty$ for any $v \in {\mathbb R}^n$ and the second
inequality uses $\| \Phi w^* \|_2  = \sqrt{(w^*)^T \Phi ^T \Phi
w^* }  \ge \sqrt {(w^* )^T \lambda_{\min}(\Phi^T \Phi)w^*}=\sqrt
{\lambda_{\min} (\Phi^T \Phi )} \left\| w^*\right\|_2$. Combining
the two relations completes the proof. $\blacksquare$

The first bound easily follows from $\bar w^*= {\bf 1}_N \otimes w^*$
and the {\bf Claim}. Since $\bar v^* = 0$
from~\cref{prop:stationary-points}, the second inequality is
obvious. For the third bound, we use the expression for $\bar
\theta^*$ in~\cref{prop:stationary-points} to prove
\begin{align*}
\| \bar\Phi^T \bar \theta^* \|_\infty = & \left\| \bar \Pi \left(
-\bar r^\pi + {\bf 1} \otimes \frac{1}{N}\sum_{i=1}^N {r_i^{\pi _i
} }\right) \right\|_\infty\\
\le & \left\| \bar \Pi \left( -\bar r^\pi + {\bf 1} \otimes
\frac{1}{N}\sum_{i = 1}^N {r_i^{\pi_i}} \right)
\right\|_2\\
\le & \frac{1}{\sqrt \xi}\left\| \bar \Pi \left(-\bar r^\pi
+ {\bf 1} \otimes \frac{1}{N}\sum_{i = 1}^N {r_i^{\pi _i }} \right) \right\|_D\\
\le & \frac{1}{\sqrt \xi}\left\|  - \bar r^\pi + {\bf 1}
\otimes \frac{1}{N}\sum_{i = 1}^N {r_i^{\pi_i}} \right\|_D\\
\le & \frac{1}{\sqrt \xi}\left\|-\bar r^\pi + {\bf 1}
\otimes \frac{1}{N}\sum_{i=1}^N {r_i^{\pi_i}} \right\|_2\\
\le & \frac{1}{\sqrt \xi}\sqrt {N|{\mathcal S}|} \left\| - \bar r^\pi
+ {\bf 1}\otimes \frac{1}{N}\sum_{i=1}^N {r_i^{\pi_i}
} \right\|_\infty\\
\le & \frac{2\sigma \sqrt{N|{\mathcal S}|}}{\sqrt \xi},
\end{align*}
where the first inequality follows from $\| \cdot \|_\infty \leq
\| \cdot \|_2$, the third inequality follows from the nonexpansive
property of the projection (see~\cite[Proof of Prop.~6.9.,
pp.~355]{bertsekas1996neuro} for details), and the fact that $\| v
\|_2  \le \sqrt{n} \left\| v \right\|_\infty$ for any $v \in
{\mathbb R}^n$ is used in the fifth inequality. Lower bounds on
$\| \bar\Phi^T\bar\theta ^* \|_\infty$ are obtained as
\begin{align*}
\| \bar \Phi^T \bar \theta^* \|_\infty \ge& \frac{1}{\sqrt
{N|{\mathcal S}|}} \| \bar \Phi ^T \bar \theta^* \|_2\\
\ge& \frac{1}{\sqrt {N|{\mathcal S}|}}\sqrt{\lambda_{\min}
(\bar \Phi ^T \bar \Phi)} \| \bar \theta^* \|_2\\
=& \sqrt {\frac{\lambda_{\min}(\Phi^T \Phi )}{|{\mathcal S}|}} \| \bar
\theta ^* \|_2 \\
\ge& \sqrt {\frac{\lambda_{\min } (\Phi^T
\Phi)}{|{\mathcal S}|}} \| \bar\theta^* \|_\infty.
\end{align*}
Combining the two inequalities yields the third bound. For the
last inequality, we use~\cref{prop:saddle-point-mu} and obtain a
bound on $\bar {\bf L}^\dag \bar \Phi ^T (I-\gamma \bar P^\pi)^T
\bar D\bar \Phi \bar \theta^* \in \tilde {\mathcal F}^*$
\begin{align*}
\| \bar \mu \|_\infty =& \| \bar {\bf L}^\dag \bar \Phi^T (I -
\gamma \bar P^\pi)^T \bar D\bar \Phi \bar\theta^*
\|_\infty\\
\le& \| \bar {\bf L}^\dag \|_\infty \| \bar\Phi^T (I - \gamma\bar
P^\pi)^T \bar D\bar \Phi \|_\infty
\| \bar \theta^* \|_\infty\\
\le& \| {\bf L}^\dag \|_\infty \| \Phi \|_\infty^2 \| (I-\gamma
P^\pi)^T D \|_\infty \| \bar
\theta^* \|_\infty\\
\le& |{\mathcal S}| \| {\bf L}^\dag \|_\infty \| \Phi \|_\infty ^2
\| \bar \theta^* \|_\infty \\
\le& |{\mathcal S}| \| {\bf L}^\dag \|_\infty  \| \Phi \|_\infty^2
2\sigma |{\mathcal S}|\sqrt {\frac{N}{\xi \lambda_{\min} (\Phi^T\Phi
)}},
\end{align*}
where the third inequality follows from the fact that absolute
values of all elements of $(I-\gamma P^\pi)^T D$ are less than
one, and the fourth inequality uses the bounds on
$\|\bar\theta^*\|_\infty$.
\end{proof}

In this section, we analyzed the set of saddle-points corresponding to the MARLP. In the next section, we introduce the proposed multi-agent RL algorithm, which solves the saddle-point problem of the MARLP in~\eqref{eq:saddle} by using the stochastic primal-dual algorithm.

\section{Primal-dual distributed GTD algorithm (primal-dual DGTD)}

In this section, we study a distributed GTD algorithm to
solve~\cref{problem:multi-agent-RL}. The main idea is to solve the saddle-point problem of the MARLP in~\eqref{eq:saddle} by using the stochastic primal-dual algorithm, where the unbiased stochastic gradient estimates are obtained by using samples of the state, action, and reward. To proceed, we first modify the saddle-point problem of the MARLP in~\eqref{eq:saddle} to a constrained saddle-point problem whose domains are confined to compact sets.

\cref{lemma:bound-lemma3} provides rough estimates of the bounds
on the sets that include at least one saddle-point of the
Lagrangian function~\eqref{eq:Lagrangian1}. Define the cube
$B_\beta:= \{ x \in {\mathbb R}^{|{\mathcal S}|N}: \| x \|_\infty \le
\beta \}$ and $C_{\bar \theta}= B_{c_{\bar \theta}+\beta _{\bar
\theta}} ,C_{\bar v} =B_{c_{\bar v} + \beta_{\bar v}} ,C_{\bar \mu
}  = B_{c_{\bar \mu }+\beta_{\bar \mu}} ,C_{\bar w}= B_{c_{\bar w}
+\beta_{\bar w}}$ for $\beta_{\bar\theta} ,\beta_{\bar v}
,\beta_{\bar \mu},\beta _{\bar w}>0$. Then, the constraint sets
satisfy $\bar\theta^* \in C_{\bar\theta}$, $\bar v^*\in C_{\bar
v}$, $\bar w^*\in C_{\bar w}$, and $C_{\bar \mu}\cap{\mathcal
F}^*\neq\emptyset$. Estimating $c_{\bar\theta} ,c_{\bar v}
,c_{\bar \mu},c_{\bar w}>0$ requires additional analysis or is
almost infeasible in most real applications. However, in practice,
we can consider sufficiently large parameters $c_{\bar\theta }
,c_{\bar v} ,c_{\bar \mu } ,c_{\bar w}>0$ so that they include at
least one solution. With this respect, we assume that sufficiently
large sets $C_{\bar \theta },C_{\bar v},C_{\bar \mu },C_{\bar w}$
satisfy $C_{\bar \mu}\cap{\mathcal F}^*\neq\emptyset$. For simpler analysis, we also assume that the solutions are included in interiors of the compact sets.
\begin{assumption}\label{assumption:constrain-set}
The constraint sets satisfy $\bar\theta^* \in C_{\bar\theta}$,
$\bar v^*\in C_{\bar v}$, $\bar w^*\in C_{\bar w}$, and $C_{\bar
\mu}\cap{\mathcal F}^*\neq\emptyset$.
\end{assumption}

Under~\cref{assumption:constrain-set}, finding a saddle-point
in~\eqref{eq:saddle} can be reduced to the constrained
saddle-point problem
\begin{align*}
&\mathop {\min}_{\bar \theta ,\bar v,\bar \mu } \mathop
{\max}_{\bar w} {\mathcal L}(\bar \theta ,\bar v,\bar \mu ,\bar w)\\
&{\rm subject}\,\,{\rm to}\quad \bar w \in C_{\bar w} ,\quad (\bar
\theta ,\bar v,\bar\mu) \in C_{\bar\theta}\times C_{\bar v} \times
C_{\bar\mu}.
\end{align*}
For notational convenience, introduce the notation
\begin{align*}
&\bar x: = \begin{bmatrix}
   \bar\theta\\
   \bar v\\
   \bar\mu \\
\end{bmatrix},\quad \bar x^*: = \begin{bmatrix}
   \bar \theta^* \\
   \bar v^*  \\
   \bar \mu^*  \\
\end{bmatrix},\\
&{\mathcal W}:= C_{\bar w},\quad {\mathcal X}:=C_{\bar\theta}\times
C_{\bar v}\times C_{\bar \mu }.
\end{align*}

Then, the saddle-point problem is $\min_{\bar x\in {\mathcal X}} \max_{\bar w \in {\mathcal W}}
{\mathcal L}(\bar x,\bar w)$. If the gradients of the Lagrangian are available, then the deterministic primal-dual
algorithm~\cite{nedic2009subgradient} can be used as follows:
\begin{align}
&\bar x_{k + 1}=\Gamma_{\mathcal X} (\bar x_k -\alpha_k {\mathcal
L}_{\bar x} (\bar x_k ,\bar w_k) ),\label{eq:primal-dual-form1}\\
&\bar w_{k + 1}=\Gamma_{\mathcal W} (\bar w_k+\alpha_k{\mathcal L}_{\bar
w}(\bar x_k,\bar w_k)).\label{eq:primal-dual-form2}
\end{align}

In this paper, our problem allows only stochastic gradient estimates of the Lagrangian function: the exact gradients are not available, while only their unbiased stochastic estimations are given. In this case, the stochastic primal-dual algorithm~\cite{nemirovski2009robust} introduced in~\cref{sec:primal-dual-algorithm} can find a solution under certain conditions
\begin{align*}
&\bar x_{k+1}=\Gamma_{\mathcal X} (\bar x_k-\alpha_k ({\mathcal L}_x (\bar x_k,\bar w_k)+
\varepsilon_k)),\\
&\bar w_{k+1}=\Gamma_{\mathcal W} (\bar w_k+\alpha_k ({\mathcal L}_w
(\bar x_k,\bar w_k)+\xi_k)),
\end{align*}
where $\epsilon_k$ and $\xi_k$ are i.i.d. random variables with zero mean. In our case, stochastic estimates of the Lagrangian function~\eqref{eq:Lagrangian2} can be obtained by using samples of the state, action, and reward. The overall algorithm is given in~\cref{algo:DGTD}. In~\cref{algo:line:sampling}, each agent samples the state, action, and the corresponding local reward,~\cref{algo:line:primal-update} updates the primal variable according to the stochastic gradient descent step, and~\cref{algo:line:dual-update} updates the dual variable by the stochastic gradient ascent step. \cref{algo:line:projection} projects the variables to the corresponding compact sets $C_{\bar \theta },C_{\bar v},C_{\bar \mu },C_{\bar w}$, and~\cref{algo:line:output} outputs averaged iterates over the whole iteration steps instead of the final iterates. Note that the averaged dual
variables can be computed recursively~\cite[pp.~181]{bertsekas1996neuro}.
\begin{algorithm}[h]
\caption{Distributed GTD algorithm}
\begin{algorithmic}[1]

\State Set $\kappa \geq 0$ and the step-size sequence $\{\alpha_k \}_{k=0}^\infty$.

\For{agent $i\in \{1,\ldots,N\}$}

\State Initialize $(\theta _0^{(i)} ,v_0^{(i)} ,\mu
_0^{(i)},w_0^{(i)} )$.

\EndFor

\For{$k \in \{0,\ldots,T-1\}$}

\For{agent $i\in \{1,\ldots,N\}$}\label{algo:line:sampling}

\State Sample $(s,a,s')$ with $s \sim d,a \sim \pi_i(\cdot|s),s'
\sim P(s,a,\cdot)$, $\hat r_i := \hat r_i (s,a,s')$.

\State Update primal variables according to\label{algo:line:primal-update}
\begin{align*}
\theta_{k+1/2}^{(i)}=&\theta_{k}^{(i)}-
\alpha_k [\phi\phi^T \theta_{k}^{(i)}+\phi\phi^T w_{k}^{(i)}\\
&-\gamma\phi(\phi')^T w_{k}^{(i)}-\phi \hat r_i]\\
v_{k+1/2}^{(i)} =&v_{k}^{(i)}- \alpha_k\left[v_{k}^{(i)} \right.\\
& \left.- \left(
|{\mathcal N}_i(k)|w_{k}^{(i)} -
\sum_{j\in {\mathcal N}_i(k)}w_{k}^{(j)} \right) \right],
\end{align*}
where ${\mathcal N}_i(k)$ is the neighborhood of node $i$ on the graph
${\mathcal G}(k)$, $\phi:=\phi(s),\phi':=\phi(s')$.

\State Update dual variables according to\label{algo:line:dual-update}
\begin{align*}
\mu_{k+1/2}^{(i)}=& \mu_{k}^{(i)}+\alpha_k \left( |{\mathcal N}_i(k)
|w_{k}^{(i)}-\sum_{j\in {\mathcal
N}_i(k)}w_{k}^{(j)}\right),\\
w_{k+1/2}^{(i)}=& w_{k}^{(i)} - \alpha_k \left( |{\mathcal N}_i(k)
|v_{k}^{(i)}-\sum_{j \in {\mathcal N}_i(k)}v_{k}^{(j)}\right)\\
& - \alpha_k \left( |{\mathcal N}_i(k)|\mu_{k}^{(i)}-\sum_{j\in {\mathcal
N}_i (k)}\mu_{k}^{(j)} \right)\\
&+\alpha_k(\phi\phi^T \theta_{k}^{(i)}-\gamma\phi'\phi^T \theta_{k}^{(i)})\\
&-\alpha_k \kappa \left( |{\mathcal N}_i(k)|w_{k}^{(i)}-\sum_{j\in {\mathcal
N}_i (k)}w_{k}^{(j)} \right).
\end{align*}

\State Project parameters:\label{algo:line:projection}
\begin{align*}
&\theta_{k+1}^{(i)}=\Gamma_{C_{\bar\theta}} (\theta_{k +
1/2}^{(i)}),\quad v_{k+1}^{(i)}=\Gamma_{C_{\bar v}}
(v_{k + 1/2}^{(i)}),\\
&\mu_{k + 1}^{(i)}= \Gamma_{C_{\bar\mu}} (\mu_{k+1/2}^{(i)}
),\quad w_{k + 1}^{(i)}=\Gamma_{C_{\bar w}} (w_{k+1/2}^{(i)}).
\end{align*}

\EndFor

\EndFor \State {\bf Output} The averaged $\hat
w_T^{(i)}  = \frac{1}{T}\sum_{k=0}^{T} {w_k^{(i)}},i \in {\mathcal
V}$, and last, $w_T^{(i)},i \in {\mathcal
V}$, dual iterates.\label{algo:line:output}

\end{algorithmic}
\label{algo:DGTD}
\end{algorithm}

The next proposition states that the averaged dual variable
converges to the set of saddle-points in terms of the $\varepsilon$-saddle set with a vanishing $\varepsilon$.
\begin{proposition}[Finite-time convergence~I]\label{prop:convergence-DGTD}
Consider~\cref{algo:DGTD}, assume that the step-size sequence, $(\alpha_k)_{k=0}^\infty$,
satisfies $\alpha_k =\alpha_0 /\sqrt{k+1}$ for some $\alpha_0 >0$, and let
\begin{align*}
&\bar x_k := \begin{bmatrix}
   \bar \theta_k\\
   \bar v_k\\
   \bar \mu_k\\
\end{bmatrix},\quad \bar \theta _k := \begin{bmatrix}
   \theta _k^{(1)}\\
    \vdots   \\
   \theta_k^{(N)} \\
\end{bmatrix},\quad \bar v_k := \begin{bmatrix}
   v_k^{(1)}\\
    \vdots   \\
   v_k^{(N)}\\
\end{bmatrix},\\
&\bar \mu_k := \begin{bmatrix}
   \mu_k^{(1)}\\
    \vdots   \\
   \mu_k^{(N)}\\
\end{bmatrix},\quad \bar w_k := \begin{bmatrix}
   w_k^{(1)}\\
    \vdots \\
   w_k^{(N)} \\
\end{bmatrix},
\end{align*}
and $\hat x_T = \frac{1}{T}\sum_{k=0}^{T-1}{\bar x_k }$ and $\hat
w_T = \frac{1}{T}\sum_{k=0}^{T-1} {\bar w_k }$ be the averaged
dual iterates generated by~\cref{algo:DGTD} with $T \ge 1$. Then, for any $\varepsilon > 0,\delta \in (0,1)$, Then, for any
$\varepsilon> 0,\delta>0$, if $T \ge\max \{\Omega_1,\Omega_2\}=:\omega(\varepsilon,\delta)$, then
\begin{align*}
&{\mathbb P}[(\hat x_T,\hat w_T ) \in {\mathcal H}_\varepsilon] \ge 1 - \delta,
\end{align*}
where
\begin{align*}
\Omega_1:=&\frac{8C^2((\alpha_0+2)^2 C^2+(\alpha_0+4)\varepsilon/6)}{\varepsilon ^2} \ln\left(\frac{1}{\delta}\right),\\
\Omega_2:=&\frac{4C^4 (2\alpha_0^{-1}+\alpha_0)^2}{\varepsilon^2}.
\end{align*}
\end{proposition}
\begin{proof}
Since the reward is bounded by $\sigma$, the stochastic estimates of the gradient are
bounded, and the inequalities $\| {\mathcal L}_x(x_k,w_k)+\varepsilon_k
\|_2 \le C$ and $\|{\mathcal L}_w(x_k,w_k)+\xi_k \|_2 \le C$ are satisfied from some constant $C >0$. Then, the is proved by using~\cref{prop:convergence1}.
\end{proof}

\cref{prop:convergence-DGTD} provides a convergence of the iterates of~\cref{algo:DGTD} to the $\varepsilon$-saddle set, ${\mathcal H}_\varepsilon$, with ${\mathcal O}(1/\varepsilon^2)$ samples (or ${\mathcal O}(1/\sqrt{T})$ rate). For the specific ${\mathcal L}$ for our problem, we can obtain stronger convergence results with convergence rates.
\begin{proposition}[Finite-time convergence~II]\label{prop:convergence-DGTD2}
Consider~\cref{algo:DGTD} and the assumptions in~\cref{prop:convergence-DGTD}.  Fix any $\varepsilon > 0$ and $\delta \in (0,1)$. If $T \ge \omega ((\kappa /2)\varepsilon ,\delta)$, then
\begin{align}
&{\mathbb P}[\bar w_T^T \bar L\bar w_T \le \varepsilon]\ge 1 - \delta,\label{eq:5}
\end{align}
where the function $\omega: {\mathbb R} \times {\mathbb R} \to {\mathbb R}$ is defined in~\cref{prop:convergence-DGTD}.

Moreover, if
\begin{align*}
T \ge\omega \left(\frac{\min \{\lambda_{\min} (\bar \Phi ^T \bar D\bar \Phi \bar \Phi ^T \bar D\bar \Phi ),1\}}{2\sqrt {\lambda_{\max} (\bar\Phi^T \bar D\bar\Phi\bar\Phi^T \bar D\bar\Phi+I)}}\varepsilon,\delta \right),
\end{align*}
then
\begin{align}
&{\mathbb P}[\|\bar\theta_T -\bar\theta^*\|_2^2+\|\bar v_T\|_2^2\le \varepsilon]\ge 1-\delta.\label{eq:6}
\end{align}
\end{proposition}
\begin{proof}
The proof is based on~\cref{prop:convergence-DGTD}, the strong convexity of ${\mathcal L}$ in some arguments, and the Lipschitz continuity of the gradient of ${\mathcal L}$. In particular, by~\cref{prop:convergence-DGTD}, if $T \ge \omega (\varepsilon ,\delta)$, then with probability $1-\delta$, $(\hat x_T,\hat w_T)\in {\mathcal H}_\varepsilon$, meaning that
\begin{align}
&{\mathcal L}(\bar\theta_T,\bar v_T,\bar \mu_T,\bar w) - {\mathcal L}(\bar\theta,\bar v,\bar\mu,\bar w_T)\le \varepsilon.\label{eq:4}
\end{align}
holds for all $\bar w \in {\mathcal W},(\bar \theta ,\bar v,\bar \mu ) \in {\mathcal X}$. Setting $\bar w = \bar w^* ,(\bar\theta,\bar v,\bar \mu ) = (\bar \theta ^* ,\bar v^* ,\bar \mu ^* )$ in~\eqref{eq:4} and using the definition of the saddle-point, we have
$\varepsilon \ge {\mathcal L}(\bar\theta_T,\bar v_T ,\bar\mu_T,\bar w^*)- {\mathcal L}(\bar\theta^*,\bar v^*,\bar\mu^* ,\bar w_T)\ge {\mathcal L}(\bar\theta^*,\bar v^* ,\bar\mu^*,\bar w^* ) - {\mathcal L}(\bar\theta^*,\bar v^* ,\bar\mu ^*,\bar w_T)= \frac{\kappa }{2}\bar w_T^T \bar L\bar w_T$, where the second inequality is due to~\cref{def:saddle-point} and the first equality follows by using the definition~\eqref{eq:Lagrangian2} and the KKT condition~\eqref{eq:KKT-points}. Replacing $\varepsilon$ with $(\kappa/2)\varepsilon $ yields the first result. Moreover, setting $\bar w=\bar w^*,(\bar\theta,\bar v,\bar\mu)=(\bar\theta^*,\bar v^*,\bar\mu^* )$ in~\eqref{eq:4} and using the definition of the saddle-point, we have
$\varepsilon \ge {\mathcal L}(\bar\theta_k ,\bar v_k ,\bar\mu_k,\bar w^*)-{\mathcal L}(\bar\theta^*,\bar v^*,\bar\mu^*,\bar w_k) \ge {\mathcal L}(\bar\theta_k,\bar v_k,\bar\mu_k,\bar w^*)-{\mathcal L}(\bar\theta^*,\bar v^*,\bar\mu^*,\bar w^*)= f(\bar\theta_k,\bar v_k,\bar\mu_k ) - f(\bar\theta^*,\bar v^*,\bar\mu^*)$, where $f(\cdot,\cdot,\cdot ) = {\mathcal L}( \cdot , \cdot , \cdot ,\bar w^* )$. It is easily prove that $f$ has a Lipschitz gradient with parameter $\sqrt {\lambda_{\max} (\bar\Phi^T \bar D\bar \Phi \bar \Phi^T \bar D\bar \Phi+I)}$, i.e., .
\begin{align*}
&\|\nabla f(\bar\theta,\bar v,\bar\mu)-\nabla f(\bar\theta',\bar v',\bar\mu')\|_2\\
&\le \sqrt{\lambda_{\max}(\bar\Phi^T \bar D\bar \Phi\bar\Phi^T \bar D\bar \Phi+I)} \left\| \begin{bmatrix}
   \bar\theta -\bar \theta'\\
   \bar v - \bar v'\\
   \bar\mu  - \bar \mu'\\
\end{bmatrix} \right\|_2
\end{align*}

Therefore, using~\cite[Prop.~6.1.9]{bertsekas2015convex} and using the fact that $(\bar\theta^*,\bar v^*,\bar\mu^*)$ is a minimizer of $f$, one concludes $\frac{1}{2\sqrt{\lambda_{\max}(\bar\Phi^T \bar D\bar \Phi \bar \Phi ^T \bar D\bar \Phi  + I)}}\|\nabla {\mathcal L}(\bar \theta _k ,\bar v_k ,\bar \mu _k ,\bar w^*\|_2^2\le \varepsilon$. After algebraic manipulations with~\eqref{eq:KKT-points}, we obtain $\frac{\min \{\lambda_{\min}(\bar\Phi^T\bar D\bar \Phi \bar\Phi ^T \bar D\bar\Phi ),1\}}{2\sqrt {\lambda_{\max} (\bar \Phi ^T \bar D\bar\Phi\bar \Phi ^T \bar D\bar \Phi  + I)}}(\| \bar \theta  - \bar \theta^*\|_2^2  + \| \bar v\|_2^2 ) \le \varepsilon$. The second result is obtained by replacing $\varepsilon$ with $\frac{\min \{\lambda _{\min } (\bar \Phi ^T \bar D\bar \Phi \bar \Phi ^T \bar D\bar \Phi ),1\}}{2\sqrt{\lambda_{\max}(\bar \Phi ^T \bar D\bar \Phi \bar \Phi ^T \bar D\bar \Phi+I)}}\varepsilon$.
\end{proof}

The first result in~\eqref{eq:5} implies that the iterate, $\bar w_T$, reaches a consensus with at most ${\mathcal O}(1/\varepsilon^2)$ samples or at ${\mathcal O}(1/\sqrt{T})$ rate. Similarly,~\eqref{eq:6} implies that the squared norm of the errors of $\bar\theta_T$ and $\bar v_T$, $\|\bar\theta_T -\bar\theta^*\|_2^2+\|\bar v_T\|_2^2$, converges at ${\mathcal O}(1/\sqrt{T})$ rate. However,~\eqref{eq:5} does not suggest anything about the convergence rate of $\|\bar w_T -\bar w^*\|_2^2$ and $\|\bar \mu_T -\bar \mu^*\|_2^2$. Still, their asymptotic convergence is guaranteed by~\cref{prop:convergence-DGTD}. The main reason is the lack of the strong convexity with respect to these variables. However, we can resolve this issue with a slight modification of the algorithm by adding the regularization term $(\rho/2)\bar \mu^T \bar \mu -(\rho/2)\bar w^T \bar w$ to the Lagrangian ${\mathcal L}$ with a small $\rho > 0$ so that ${\mathcal L}(\bar\theta,\bar v,\bar\mu)$ is strongly convex in $\bar\theta$ and strongly concave in $\bar\mu$. In this case, the corresponding saddle-points are slightly altered depending on $\rho$.

\begin{remark}
\cref{prop:convergence-DGTD} and~\cref{prop:convergence-DGTD2} apply the analysis of the primal-dual algorithm in~\cref{prop:convergence1}, and exhibit ${\mathcal O}(1/\sqrt{T})$ convergence rate. The recent primal-dual algorithm in~\cite{ding2019fast} has faster ${\mathcal O}(1/T)$ rate, and can be applied to solve the saddle-point problem in~\eqref{eq:saddle}.
\end{remark}
\begin{remark}
The last line of~\cref{algo:DGTD} indicates that both the averaged iterates, $\hat
w_T^{(i)}=\frac{1}{T}\sum_{k=0}^{T} {w_k^{(i)}},i \in {\mathcal
V}$, and the last iterate, $w_T^{(i)},i \in {\mathcal
V}$, can be used for estimates of the solution. The result in~\cite{lee2018primal} proves the asymptotic convergence of the last iterate of~\cref{algo:DGTD} by using the stochastic approximation method~\cite{kushner2003stochastic}. For the convergence, the step-size rules should satisfy $\alpha_k >0,\alpha_k \to 0,\sum_{k=0}^\infty {\alpha_k} =
\infty,\sum_{k=0}^\infty
{\alpha_k^2}<\infty$, called the Robbins-Monro rule. An example is $\alpha_k=\alpha_0/(k+\beta)$ with $\alpha_0,\beta >0$. On the other hand,~\cref{prop:convergence-DGTD} and~\cref{prop:convergence-DGTD2} prove the convergence of the averaged iterate of~\cref{algo:DGTD} with convergence rates by using tools in optimization. The step-size rule is $\alpha_k=\alpha_0\sqrt{k+\beta}$ with $\alpha_0,\beta >0$, which does not obey the Robbins-Monro rule.
\end{remark}
\begin{remark}
There exist several RLs based on stochastic primal-dual approaches. The GTD can be interpreted as a stochastic primal-dual
algorithm by using Lagrangian duality theory~\cite{macua2015distributed}. The work in~\cite{chen2016stochastic} proposes
primal-dual reinforcement learning algorithm for the single-agent policy optimization problem, where a linear
programming form of the MDP problem is solved. A primal-dual
algorithm variant of the GTD is investigated in~\cite{mahadevan2014proximal} for a single-agent RL problem.
\end{remark}
\begin{remark}
When nonlinear function approximation is used, convergence to a global optimal solution is hardly guaranteed in general. In particular, for minimization problems, stochastic gradients converge to a local stationary point~\cite{ghadimi2013stochastic}. On the other hand, convergence of stochastic primal-dual algorithms to a saddle-point for general non-convex min-max problems is still an open problem~\cite{lin2018solving}. In this respect, the convergence of our algorithm with general nonlinear function approximation is a challenging open question, which needs significant efforts in the future.
\end{remark}

\section{Simulation}

In this section, we provide simulation studies that illustrate potential
applicability of the proposed approach.
\begin{example}
In this example, we provide a comparative analysis using simulations. We consider the Markov chain
\begin{align*}
&P^\pi =\begin{bmatrix}
   0.1 & 0.5 & 0.2 & 0.2 \\
   0.5 & 0.0 & 0.1 & 0.4 \\
   0.0 & 0.9 & 0.1 & 0.0 \\
   0.2 & 0.1 & 0.1 & 0.6  \\
\end{bmatrix},
\end{align*}
where $\pi$ is not explicitly specified, $|{\mathcal S}|=4$, $\gamma = 0.8$, feature vector $\phi (s) = \begin{bmatrix}
   \exp(-s^2)\\\exp(-(s-4)^2)\\
\end{bmatrix}$, local expected reward functions
\begin{align*}
&r_1^\pi=\begin{bmatrix}
   0 & 0 & 0 & 50\\
\end{bmatrix}^T,\quad r_2^\pi=\begin{bmatrix}
   0 & 0 & 0 & 0\\
\end{bmatrix}^T,\\
&r_3^\pi = \begin{bmatrix}
   0 & 0 & 0 & 0\\
\end{bmatrix}^T,\quad r_4^\pi= \begin{bmatrix}
   0 & 0 & 0 & 0\\
\end{bmatrix}^T,\\
&r_5^\pi= \begin{bmatrix}
   0 & 0 & 0 & 0\\
\end{bmatrix}^T,
\end{align*}
and the five RL agents over the network given in~\cref{fig:graph}.
\begin{figure}[h]
\centering\epsfig{figure=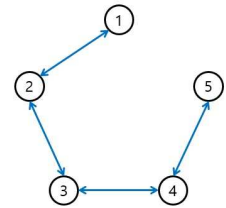,width=3cm} \caption{Network topology of five RL agents.}\label{fig:graph}
\end{figure}
\begin{figure}[h!]
\centering\epsfig{figure=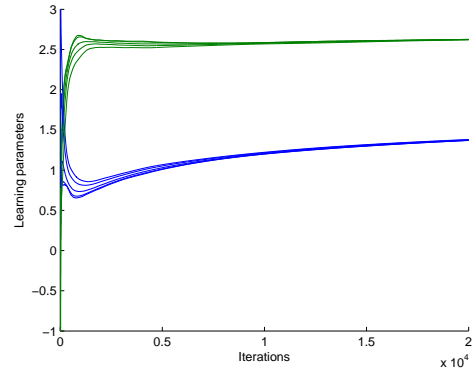,width=7cm}
\caption{Evolution of iterates of the proposed DGTD (solid lines with different colors for different parameters),~\cref{algo:DGTD}. We use the step-size rule $\alpha_k=10/\sqrt {k+100}$ and $\kappa = 1$.}\label{fig:comp1}
\end{figure}

\cref{fig:comp1} depicts evolutions of two parameter iterates of the proposed DGTD (different colors for different parameters),~\cref{algo:DGTD}. It shows that the parameters of five agents reach a consensus and converge to certain numbers. The results empirically demonstrate the proposed DGTD.
\end{example}

\begin{example}\label{ex:ex1}
Consider a $20[{\rm m}]\times 20[{\rm m}]$ continuous space $\mathcal
X$ with three robots (agent~1 (blue), agent~2 (red), and agent~3
(black)), which patrol the space with identical stochastic motion
planning policies $\pi_1 = \pi_2 = \pi_3 =\pi$. We consider a
single integrator system for each agent $i$: $\dot x_i(t) =
u_i(t)$ with the control policy $u_i(t) = -h(x_i(t) - r_i)$
employed from~\cite{panagou2014decentralized}, where $t \in
{\mathbb R}_+$ is the continuous time, $h \in {\mathbb R}_{++}$ is
a constant, $r_i$ is a randomly chosen point in $\mathcal X$ with
uniform distribution over $\mathcal X$. Under the control policy
$u_i(t)=-h(x_i(t)-r_i)$, $x_i(t)$ globally converges to $r_i$ as
$t \to \infty$~\cite[Lemma~1]{panagou2014decentralized}. When
$x_i$ is sufficiently close to the destination $r_i$, then it
chooses another destination $r_i$ uniformly in $\mathcal X$, and all
agents randomly maneuver the space $\mathcal X$. The continuous space
$\mathcal X$ is discretized into the $20\times 20$ grid world $\mathcal
S$. The collaborative objective of the three
robots is to identify the dangerous region using individually collected
reward (risk) information by each robot. The global value function
estimated by the proposed distributed GTD learning informs the
location of the points of interest. The three robots maneuver the
space and detect the dangers together. For instance, these regions
represent those with frequent turbulence in  commercial flight
routes or enemies in battle fields. Each robot is equipped with a
different sensor that can detect different regions, while a pair
of robots can exchange their parameters, when the distance between
them is less than or equal to $5$. We assume that robots do not
interfere with each other; thereby we can consider three
independent MDPs with identical transition models.
\begin{figure}[h!]
\centering\epsfig{figure=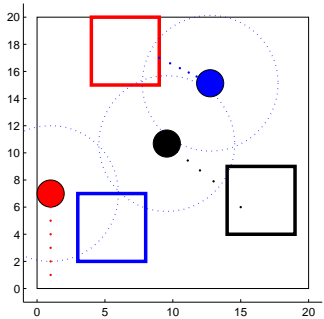,width=6.5cm} \caption{Three
dangerous regions that can be detected by three different
UAVs.}\label{fig:ex1a}
\end{figure}
The three regions and robots are depicted in~\cref{fig:ex1a}, where
the blue region is detected only by agent~1 (blue circle), the red
region is detected only by  agent~2 (red circle), and the black
region only by  agent~3 (black circle).

For each agent, the detection occurs only if the UAV flies over
the region, and a reward $\hat r = 100$ is given in this case.
In the  scenario above, the reward is
given,  when  turbulence is detected:
\cref{algo:DGTD} is applied with $\gamma = 0.5$ and
$\Phi=I_{|{\mathcal S}|}$ (tabular representation). We
run~\cref{algo:DGTD} with 50000 iterations, and the results are
shown in~\cref{fig:ex1b}. The results suggest that all agents
successfully estimate identical value functions, which are aware
of three regions despite of the incomplete sensing abilities and
communications.  The obtained value function can be
used to design a motion planning policy to travel safer routes.
\begin{figure}[h!]
\centering \subfigure{\epsfig{figure=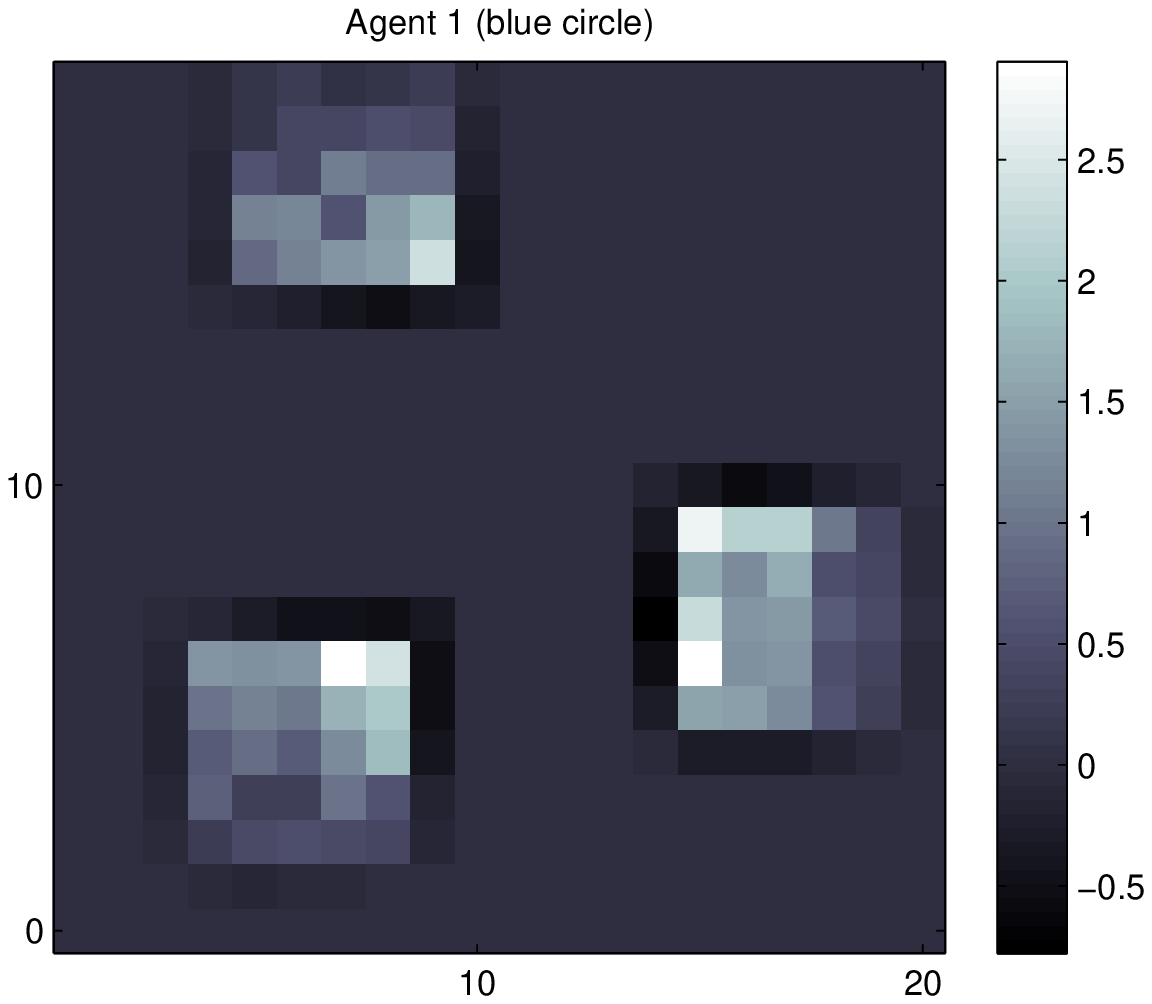,width=4cm}}
\subfigure{\epsfig{figure=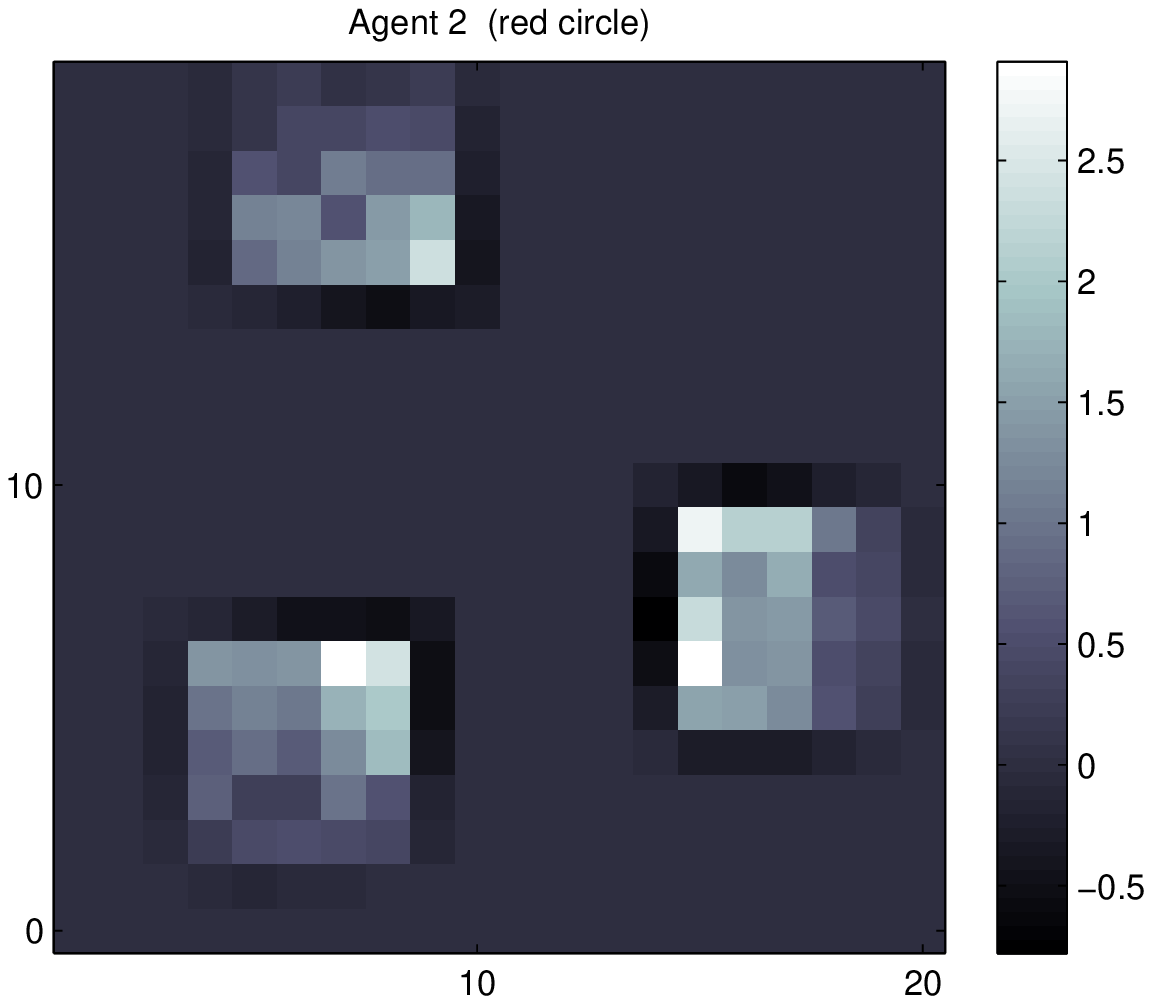,width=4cm}}
\subfigure{\epsfig{figure=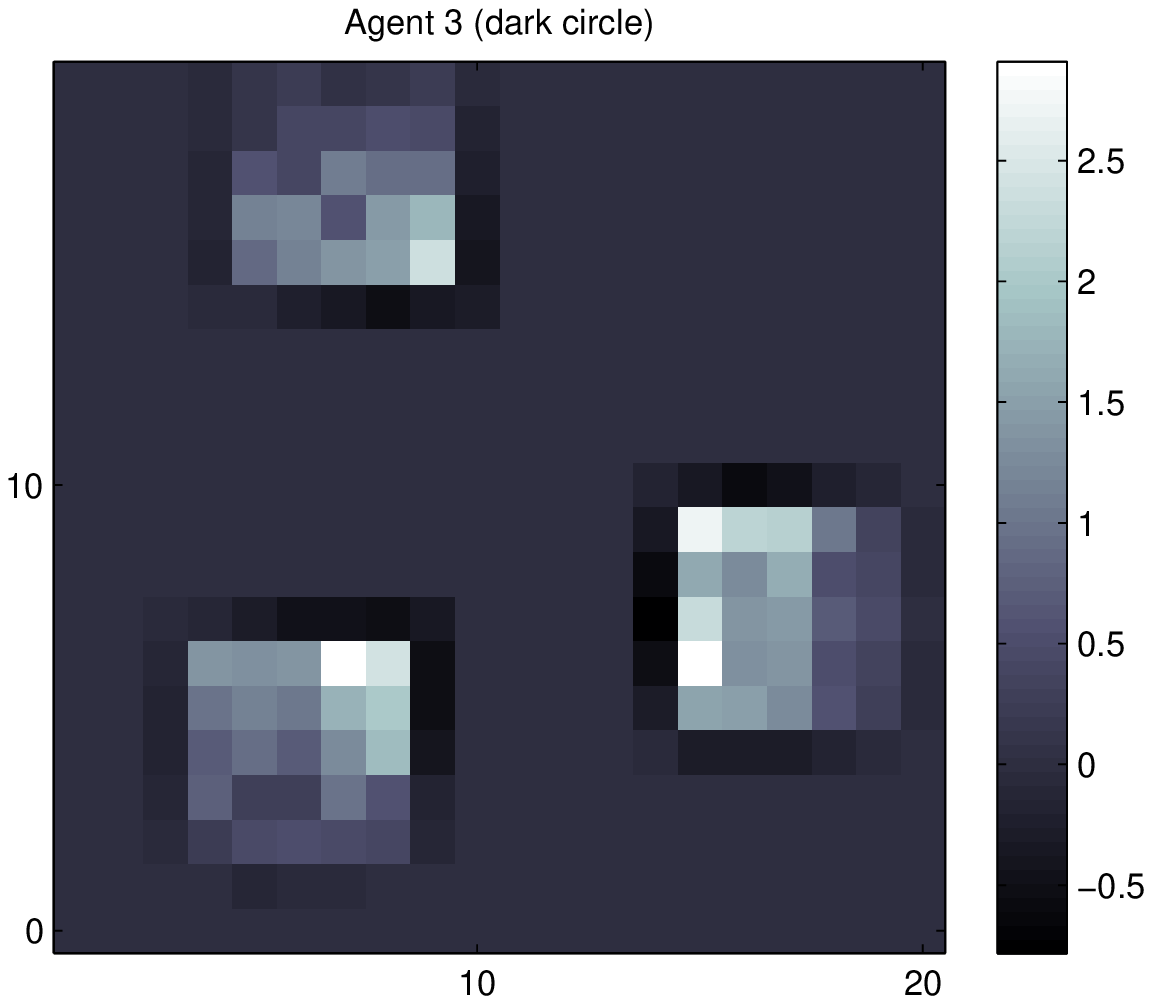,width=4cm}}
\caption{\cref{ex:ex1}. 2D plots of value functions of three
different agents.}\label{fig:ex1b}
\end{figure}
\end{example}

\section{Conclusion}
In this paper, we study a distributed GTD learning for
multi-agent MDPs using a stochastic primal-dual algorithm. Each agent receives local reward through a
local processing, while information exchange over random
communication networks allows them to learn the global value
function corresponding to a sum of local rewards. Possible future
research includes its extension to actor-critic and
Q-learning algorithms.

\section*{Acknowledgement}
D. Lee is thankful to N. Hovakimyan and H. Yoon for their fruitful comments on this paper.

\bibliographystyle{IEEEtran}
\bibliography{reference}

\appendix

\begin{center}
{\Large Appendix}
\end{center}

\section{Proof of~\cref{prop:convergence1}}
In this section, we will provide a proof of~\cref{prop:convergence1}. We begin with a basic technical lemma.
\begin{lemma}[Basic iterate relations~\cite{nedic2009subgradient}]\label{lemma:basic-iterate-relations1}
Let the sequences $(x_k,w_k )_{k=0}^\infty$ be generated by the
stochastic subgradient algorithm
in~\eqref{eq:promai-dual-subgrad1}
and~\eqref{eq:promai-dual-subgrad2}. Then, we have:
\begin{enumerate}
\item For any $x \in {\mathcal X}$ and for all $k \geq 0$,
\begin{align*}
&{\mathbb E}[\|x_{k+1}-x\|^2 |{\mathcal F}_k ]\\
\le& \|x_k  - x \|^2 + \alpha_k^2 {\mathbb E}[\| {\mathcal L}_x(x_k,w_k)+\varepsilon_k \|^2 |{\mathcal F}_k ]\\
&-2\alpha_k ({\mathcal L}(x_k ,w_k ) - {\mathcal L}(x,w_k )).
\end{align*}

\item For any $w \in {\mathcal W}$ and for all $k \geq 0$,
\begin{align*}
&{\mathbb E}[\| w_{k + 1}- w \|^2 |{\mathcal F}_k ]\\
\le& \| w_k-w \|^2 + \alpha_k^2 {\mathbb E}[\|{\mathcal L}_w(x_k,w_k ) + \xi
_k \|^2 |{\mathcal F}_k ]\\
&+ 2\alpha_k ({\mathcal L}(x_k,w_k ) - {\mathcal L}(x_k ,w)).
\end{align*}
\end{enumerate}
\end{lemma}
\begin{proof}
The result can be obtained by the iterate relations
in~\cite[Lemma~3.1]{nedic2009subgradient} and taking the
expectations.
\end{proof}
\begin{lemma}[Berstein inequality for
Martingales~\cite{bercu2015concentration}]\label{lemma:Bernstein-inequality}
Let $({\mathcal M}_T )_{T = 0}^\infty$ be a square integrable
martingale such that ${\mathcal M}_0  = 0$. Assume that $\Delta {\mathcal
M}_T  \le b,\forall T \ge 1$ with probability one, where $b>0$ is
a real number and $\Delta {\mathcal M}_T $ is the Martingale
difference defined as $\Delta {\mathcal M}_T  := {\mathcal M}_T  - {\mathcal
M}_{T - 1} ,T \ge 1$. Then, for any $\varepsilon  \in [0,b]$ and
$a > 0$,
\begin{align*}
&{\mathbb P}\left[ \frac{1}{T}{\mathcal M}_T  \ge \varepsilon
,\frac{1}{T}\left\langle {\mathcal M} \right\rangle _T  \le a\right]
\le \exp \left(- \frac{T\varepsilon^2}{2(a+b\varepsilon
/3)}\right),
\end{align*}
where
\begin{align*}
\langle {\mathcal M} \rangle_T :=&\sum_{k=0}^{T-1} {\mathbb E}[\Delta {\mathcal M}_{k+1}^2 |{\mathcal F}_k].
\end{align*}
\end{lemma}

For any $x \in {\mathcal X}$ and $y \in {\mathcal Y}$, define
\begin{align*}
&{\mathcal E}_k^{(1)}(x):=\|x_k - x\|_2^2,\\
&{\mathcal E}_k^{(2)}(w):= \|w_k - w\|_2^2
\end{align*}
and
\begin{align*}
&H_k(x):=\frac{1}{2\alpha_k}({\mathcal E}_k^{(1)}(x) - {\mathbb
E}[{\mathcal E}_{k+1}^{(1)}(x)|{\mathcal F}_k]),\\
&R_k(w):= \frac{1}{2\alpha_k}({\mathcal E}_k^{(2)}(y)-{\mathbb
E}[{\mathcal E}_{k+1}^{(2)}(y)|{\mathcal F}_k]),
\end{align*}

We use ${\mathbb E}[\|{\mathcal L}_x(x_k,w_k)+\varepsilon_k\|_2^2 |{\mathcal F}_k] \le C^2$ and rearrange terms in~\cref{lemma:basic-iterate-relations1} to have
\begin{align}
&{\mathcal L}(x_k,y_k)-{\mathcal L}(x,w_k)\nonumber\\
\le& \underbrace {\frac{1}{2\alpha_k}({\mathcal E}_k^{(1)}(x) - {\mathbb E}[{\mathcal
E}_{k+1}^{(1)}(x)|{\mathcal F}_k])}_{=:H_k(x)} + \frac{\alpha_k}{2}
C^2,\nonumber\\
&\forall x \in {\mathbb R}^{|{\mathcal S}||{\mathcal A}|} \times {\mathbb R}^{|{\mathcal S}|},\label{eq:basic1}\\
&-\underbrace {\frac{1}{2\alpha_k}({\mathcal E}_k^{(2)}(w)-{\mathbb
E}[{\mathcal E}_{k+1}^{(2)}(w)|{\mathcal F}_k ])}_{=:R_k(w)} -
\frac{\alpha_k}{2} C^2 \nonumber\\
\le& {\mathcal L}(x_k,w_k)-{\mathcal L}(x_k,w),\quad
\forall w \in {\mathbb R}_+^{|{\mathcal S}||{\mathcal A}|} \times {\mathbb R}^{|{\mathcal S}||{\mathcal A}|}.\label{eq:basic2}
\end{align}

Adding these relations over $k=0,\ldots,T-1$, dividing by $T$, and
rearranging terms, we have
\begin{align}
&-\frac{1}{T}\sum_{k=0}^{T-1}{R_k(w)}-\frac{1}{T}\sum_{k=0}^{T-1}
{\frac{\alpha_k}{2}C^2}\nonumber\\
\le& \frac{1}{T}\sum_{k=0}^{T-1}{({\mathcal L}(x_k,w_k)-{\mathcal L}(x_k,w))},\quad
\forall w \in {\mathcal W}.\label{eq:eq1}
\end{align}
Similarly, we have from~\eqref{eq:basic1}
\begin{align}
&\frac{1}{T}\sum_{k=0}^{T-1}{({\mathcal L}(x_k,w_k)-{\mathcal L}(x,w_k))}\nonumber\\
\le& \frac{1}{T}\sum_{k=0}^{T-1}{H_k(x)}+\frac{1}{T}\sum_{k=0}^{T-1}
{\frac{\alpha_k}{2} C^2} ,\quad \forall x \in {\mathcal X}.\label{eq:eq2}
\end{align}
Multiplying both sides of~\eqref{eq:eq1} by $-1$ and adding it
with~\eqref{eq:eq2} yields
\begin{align*}
&\frac{1}{T}\sum_{k=0}^{T-1}{({\mathcal L}(x_k,w) - {\mathcal L}(x,w_k))} \\
\le& \frac{1}{T}\sum_{k=0}^{T-1} {R_k(w)} +\frac{1}{T}\sum_{k=0}^{T-1} {H_k (x)}  + \frac{C^2}{T}\sum_{k = 0}^{T - 1} {\alpha_k }.
\end{align*}

Using the convexity of ${\mathcal L}$ with respect to the first
argument and the concavity of ${\mathcal L}$ with respect to the second
argument, it follows from the last inequality that
\begin{align*}
&{\mathcal L}(\hat x_T,w)-{\mathcal L}(x,\hat w_T)\\
\le&\frac{1}{T}\sum_{k=0}^{T-1} {R_k(w)} + \frac{1}{T}\sum_{k=0}^{T-1} {H_k (x)} +\frac{C^2}{T}\sum_{k=0}^{T-1} {\alpha_k}.
\end{align*}

To proceed, we rearrange terms in the last inequality to
have
\begin{align}
&{\mathcal L}(\hat x_T,w)-{\mathcal L}(x,\hat w_T)\nonumber\\
\le& \frac{1}{T}\Phi_1(x) +\frac{1}{T}\Phi_2(y)+\frac{1}{T}{\mathcal
M}_T+ \frac{C^2}{T}\sum_{k=0}^{T-1}{\alpha_k},\nonumber\\
&\forall x\in {\mathcal X}, w\in {\mathcal W},\label{eq:duality-gap-bound2}
\end{align}
where
\begin{align*}
\Phi_1(x):=&\sum_{k=0}^{T-1}{\frac{1}{2\alpha_k}({\mathcal
E}_k^{(1)}(x)-{\mathcal E}_{k+1}^{(1)}(x))},\\
\Phi_2(w):=&\sum_{k=0}^{T-1}{\frac{1}{2\alpha_k}({\mathcal
E}_k^{(2)}(w)-{\mathcal E}_{k+1}^{(2)}(w))},\\
{\mathcal M}_T:=&\sum_{k=0}^{T-1}\frac{1}{2\alpha_k}({\mathcal
E}_{k+1}^{(1)}(x)+{\mathcal E}_{k+1}^{(2)}(w)\\
&- {\mathbb E}[{\mathcal E}_{k+1}^{(1)}(x)|{\mathcal F}_k]- {\mathbb E}[{\mathcal
E}_{k+1}^{(2)}(w)|{\mathcal F}_k]).
\end{align*}

As a next step, we derive bounds on the terms $\Phi_1(x)$ and $\Phi_2(w)$. First, $\Phi_1(x)$ is bounded by using the chains of inequalities
\begin{align}
\Phi_1(x)&= \sum_{k=0}^{T-1}{\frac{1}{2\gamma_k}({\mathcal E}_k^{(1)}(x)-{\mathcal E}_{k+1}^{(1)}(x))}\nonumber\\
&\leq \sum_{k=0}^{T-1}{\frac{1}{2\gamma_k}({\mathcal E}_k^{(1)}(x)-{\mathcal E}_{k+1}^{(1)}(x))}  + \frac{1}{\gamma_T}{\mathcal E}_T^{(1)} (x)\nonumber\\
&=\frac{1}{2}\left(\frac{1}{\gamma_0}{\mathcal
E}_0^{(1)}(x)+\sum_{k=0}^{T-1}{\left(
\frac{1}{\gamma_{k+1}}-\frac{1}{\gamma_k} \right){\mathcal
E}_{k+1}^{(1)}(x)} \right)\nonumber\\
&\le \frac{C^2}{2}  \left(\frac{1}{\gamma_0} + \sum_{k=0}^{T-1}{\left(
\frac{1}{\gamma_{k+1}}-\frac{1}{\gamma_k}\right)}
\right)\label{eq:explain-3}\\
&= \frac{C^2}{2}{\gamma_{T}},\nonumber
\end{align}
where~\eqref{eq:explain-3} is due to ${\mathcal E}_k^{(1)}(x)\leq C^2,\forall x\in {\mathcal
X}$. Similarly, we have $\Phi_2(w) \leq \frac{C^2}{2}{\alpha_{T}}$. Combining the last inequality with~\eqref{eq:duality-gap-bound2} yields
\begin{align}
&{\mathcal L}(\hat x_T,w) - {\mathcal L}(x,\hat w_T)\le \frac{1}{T}\frac{C^2}{\gamma_{T}}
+\frac{C^2}{T}\sum_{k=0}^{T-1}{\gamma_k} +\frac{1}{T}{\mathcal M}_T,\nonumber\\
&\forall x\in {\mathcal X}, w\in {\mathcal W},\label{eq:duality-gap-bound3}
\end{align}

Plugging $\alpha_k=\alpha_0/\sqrt{k+1}$ into the first term, we have
\begin{align*}
&\frac{C^2}{T \alpha_T} = \frac{C^2 \sqrt{T+1}}{T\alpha_0} \le \frac{C^2 \sqrt T+1}{T \alpha_0} = \frac{{C^2 \sqrt T }}{T\alpha_0} + \frac{C^2}{T\alpha_0}\\
&\le \frac{C^2}{\sqrt T \alpha_0} + \frac{C^2}{\sqrt T \alpha_0} = \frac{2C^2}{\sqrt T \alpha_0}.
\end{align*}

Moreover, plugging $\alpha_k=\alpha_0/\sqrt{k+1}$ into the second term leads to
\begin{align*}
&\frac{1}{T}\sum_{k=0}^{T-1}{\alpha_k}=\frac{\alpha_0}{T}\sum_{k=1}^T
{\frac{1}{\sqrt k}} \le \frac{\alpha_0}{T}\int_0^T {\frac{1}{\sqrt
t}dt}= \frac{\alpha_0 \sqrt T}{T} = \frac{\alpha_0}{\sqrt T}.
\end{align*}

Therefore, combining the bounds yields
\begin{align}
&{\mathcal L}(\hat x_T,w)-{\mathcal L}(x,\hat w_T)\le \frac{2C^2 \alpha_0^{-1}+C^2\alpha_0}{\sqrt T} + \frac{1}{T}{\mathcal M}_T.\label{eq:7}
\end{align}

To prove ${\mathcal L}(\hat x_T,w)-{\mathcal L}(x,\hat w_T)\le \varepsilon$, it suffices to prove $\frac{2C^2 \alpha_0^{-1}+C^2\alpha_0}{\sqrt T} \le \varepsilon/2$ and $\frac{1}{T}{\mathcal M}_T \le \varepsilon/2$. By simple algebraic manipulations, we can prove that the first inequality holds if
\begin{align}
T\ge&\frac{4C^4(2\alpha_0^{-1}+\alpha_0)^2}{\varepsilon^2}.\label{eq:8}
\end{align}

To prove the second inequality with high probability, we will use the Bernstein inequality in~\cref{lemma:Bernstein-inequality}. To do so, one easily proves that ${\mathbb E}[{\mathcal M}_{T+1} |{\mathcal
F}_T ] = {\mathcal M}_T$, and hence, $({\mathcal M}_T )_{T = 0}^\infty$ is a Martingale. Moreover, we will find constants $a>0$ and $b>0$ such that $\Delta {\mathcal M}_{T
+1}:={\mathcal M}_{T+1}-{\mathcal M}_T \le b$ and $\frac{1}{T}\langle {\mathcal M}
\rangle_T\le a $. Noting that ${\mathcal M}_{T+1}-{\mathcal M}_T  = \frac{1}{2\alpha_T }({\mathcal E}_{T+1}^{(1)}-{\mathbb E}[{\mathcal E}_{T+1}^{(1)}|{\mathcal F}_T])+\frac{1}{2\alpha_T}({\mathcal E}_{T+1}^{(2)}-{\mathbb E}[{\mathcal E}_{T+1}^{(2)} |{\mathcal F}_T])$, we obtain the bounds for the first two terms
\begin{align}
&\frac{1}{2\alpha_k}({\mathcal E}_{k+1}^{(1)})- {\mathbb E}[{\mathcal
E}_{k+1}^{(1)} |{\mathcal F}_k])\\
=& \frac{1}{2\alpha_k} \|\Gamma_{\mathcal X} (x_k-\alpha_k
{\mathcal L}_x (x_k,w_k)-\alpha_k \varepsilon_k)-x^*\|^2\nonumber\\
&-\frac{1}{2\alpha_k} {\mathbb E}[\|\Gamma_{\mathcal X}(x_k-
\alpha_k {\mathcal L}_x (x_k,w_k)-\alpha_k \varepsilon_k)-x^*\|^2 |{\mathcal F}_k ]\nonumber\\
=& \frac{1}{2\alpha_k}(\|\Gamma_{\mathcal X}(x_k-\alpha_k
{\mathcal L}_x (x_k,w_k)-\alpha_k \varepsilon_k)-x_k\|_2^2+\|x_k-x^*\|_2^2\nonumber\\
&-2(x_k-x^*)^T (\Gamma_{\mathcal X}(x_k-\alpha_k {\mathcal L}_x
(x_k,w_k)-\alpha_k \varepsilon_k)-x_k)\nonumber\\
&- {\mathbb E}[\|\Gamma_{\mathcal X} (x_k-\alpha_k {\mathcal L}_x (x_k,w_k)-\alpha_k \varepsilon_k)-x_k \|_2^2 |{\mathcal F}_k]\nonumber\\
& -\|x_k-x^*\|_2^2 \nonumber\\
&-{\mathbb E}[(\Gamma_{\mathcal X}(x_k-\alpha_k{\mathcal L}_x (x_k
,w_k)-\alpha_k \varepsilon_k)-x_k)^T \nonumber\\
&\times (x_k-x^*)|{\mathcal F}_k ])\label{eq:explain-1}\\
\le& \frac{1}{2\alpha_k}\|\Gamma_{\mathcal X} (x_k-\alpha_k {\mathcal
L}_x (x_k,w_k)-\alpha_k\varepsilon_k)-x_k\|_2^2\nonumber\\
&+ \frac{1}{\alpha_k}\|x_k-x^*\|_2 \nonumber\\
&\times \| \Gamma_{\mathcal X}(x_k-\alpha_k {\mathcal L}_x (x_k,w_k)-\alpha_k\varepsilon_k)-x_k\|_2\nonumber\\
&+ \frac{1}{\alpha_k} {\mathbb E}[\| x_k-x^* \|_2 \nonumber\\
& \times \|\Gamma_{\mathcal
X}(x_k-\alpha_k {\mathcal L}_x (x_k ,w_k)-\alpha_k \varepsilon_k)-x_k \|_2 |{\mathcal F}_k ]\label{eq:explain-2}\\
\le& \frac{\alpha_k}{2} \| {\mathcal L}_x (x_k,w_k)+\varepsilon_k\|_2^2 + \|x_k-x^*\|_2 \| {\mathcal L}_x (x_k,w_k)+\varepsilon_k \|_2\nonumber\\
&+\|x_k-x^*\|_2 {\mathbb E}[\| {\mathcal L}_x (x_k,w_k)+\varepsilon_k \|_2 |{\mathcal F}_k ]\label{eq:explain-3}\\
\le& \alpha_k C^2/2+2C^2,\nonumber
\end{align}
where~\eqref{eq:explain-1} follows from the relation $\| a-b
\|_2^2 = \|a\|_2^2 + \| b \|_2^2 - 2a^T b$ for any vectors
$a,b$,~\eqref{eq:explain-2} follows from the Cauchy-Schwarz
inequality,~\eqref{eq:explain-3} follows from the nonexpansive map
property of the projection $\|\Gamma_{\mathcal X}(a)-\Gamma_{\mathcal X}(b)\|_2\le \|a-b\|_2$,
and the last inequality is obtained after simplifications.
Similarly, one gets $\frac{1}{2\alpha_k}({\mathcal E}_{k+1}^{(2)}-{\mathbb E}[{\mathcal E}_{k+1}^{(2)} |{\mathcal F}_k]) \le \alpha_k C^2/2+2C^2$. Combining the last two inequalities, we have $\Delta {\mathcal M}_{T
+1}:={\mathcal M}_{T+1}-{\mathcal M}_T  = \frac{1}{2\alpha_T
}({\mathcal E}_{T + 1}  - {\mathbb E}[{\mathcal E}_{T + 1} |{\mathcal F}_T ])
\le b$ with probability one, where $b =\alpha_0 C^2 + 4C^2$.

Next, we will prove that $\frac{1}{T}\langle {\mathcal M}
\rangle_T\le a $ holds, where $a =(\alpha _0  + 2)^2 C^4$. Using ${\mathbb E}[{\mathcal E}_k^{(1)}+ {\mathcal E}_k^{(2)} |{\mathcal F}_k ] = {\mathcal E}_k^{(1)} + {\mathcal E}_k^{(2)}$, we have
\begin{align}
&{\mathbb E}[|{\mathcal M}_{k + 1}  - {\mathcal M}_k |^2 |{\mathcal F}_k]\nonumber\\
=&\frac{1}{4\alpha_k^2} {\mathbb E}[|{\mathcal E}_{k+1}  - {\mathbb E}[{\mathcal
E}_{k+1}|{\mathcal F}_k ]|^2 |{\mathcal F}_k ]\nonumber\\
=& \frac{1}{4\alpha_k^2}{\mathbb E}[|{\mathcal E}_{k+1}-{\mathcal E}_k - {\mathbb E}[{\mathcal E}_{k+1}-{\mathcal E}_k |{\mathcal F}_k ]|^2 |{\mathcal F}_k]\nonumber\\
\le& \frac{1}{4\alpha_k^2}{\mathbb E}[{\mathbb E}[|{\mathcal E}_{k+1}-{\mathcal E}_k|^2|{\mathcal F}_k ]|{\mathcal F}_k ]\nonumber\\
=& \frac{1}{4\gamma_k^2}{\mathbb E}[|\underbrace {{\mathcal E}_{k +
1}^{(1)} + {\mathcal E}_{k+1}^{(2)}- {\mathcal E}_k^{(1)} - {\mathcal
E}_k^{(2)}}_{=:\Phi_1 }|^2 |{\mathcal F}_k ],\label{eq:eq7}
\end{align}
where the inequality follows from the fact that the variance of a
random variable is bounded by its second moment. For
bounding~\eqref{eq:eq7}, note that $\Phi_1$ is written as
\begin{align*}
\Phi_1=& \| x_{k+1}- x^*\|_2^2  - \| x_k-x^*\|_2^2\\
&+\|w_{k+1}- w^* \|_2^2-\|w_k-w^*\|_2^2.
\end{align*}
Here, the first two terms have the bound
\begin{align}
&\| x_{k+1}-x^* \|^2 -\|x_k-x^*\|^2\nonumber\\
=& \| \Gamma_{\mathcal X} (x_k  - \alpha_k L_x (x_k,w_k ) - \alpha_k
\varepsilon_k ) -x_k \|^2\nonumber\\
&-2(\Gamma_{\mathcal X} (x_k - \alpha_k L_x
(x_k ,w_k)-\alpha_k \varepsilon_k)- x_k )^T (x_k-x^*)\label{eq:explain-4}\\
\le& \| \Gamma_{\mathcal X} (x_k-\alpha_k L_x(x_k ,w_k)- \alpha_k
\varepsilon_k ) - x_k \|_2^2\nonumber\\
&+ 2 \| \Gamma_{\mathcal X} (x_k-\alpha
_k L_x(x_k,w_k) - \alpha_k \varepsilon_k) - x_k \|_2 \| x_k-x^*\|_2\label{eq:explain-5}\\
\le& \alpha_k^2 \| L_x (x_k,w_k)+\varepsilon _k \|_2^2 \nonumber\\
&+ 2\alpha_k
\| L_x(x_k,w_k)+ \varepsilon_k \|_2 \|x_k  - x^*\|_2\label{eq:explain-6}\\
\le& \alpha_k (\alpha_0 C^2+2C^2),\label{eq:explain-7}
\end{align}
where~\eqref{eq:explain-4} follows from the relation $\|a-b\|_2^2 = \|a\|_2^2+\|b\|_2^2
-2a^T b$ for any vectors $a,b$,~\eqref{eq:explain-5} follows from
the Cauchy-Schwarz inequality,~\eqref{eq:explain-6} is due to the
nonexpansive map property of the projection $\|\Gamma_{\mathcal X}
(a) - \Gamma_{\mathcal X} (b) \|_2 \le \|a-b\|_2$,~\eqref{eq:explain-7} comes from~\eqref{eq:1},~\eqref{eq:2}, and~\eqref{eq:3}.
Similarly, the second two terms in $\Phi _1$ are bounded by
$\alpha_k (\alpha_0 C^2+2C^2)$. Combining the
last two results leads to $\Phi_1  \le 2\alpha _k (\alpha_0 C^2+2C^2)$, and plugging the
bound on $\Phi_1$ into~\eqref{eq:eq7} and after simplifications,
we obtain ${\mathbb E}[|{\mathcal M}_{k+1}- {\mathcal M}_k |^2 |{\mathcal F}_k
] \le (\alpha_0+2)^2 C^4$, which is the desired conclusion.

We are now ready to apply the Bernstein inequality in~\cref{lemma:Bernstein-inequality} to prove $\frac{1}{T}{\mathcal M}_T \le \varepsilon/2$ with high probability. Fix any $x\in {\mathcal X},w\in
{\mathcal W}$ and apply the Bernstein inequality with $a$ and $b$ given above to prove
\begin{align*}
&{\mathbb P}\left[\frac{1}{T}{\mathcal M}_T\ge \frac{\varepsilon}{2}
,\frac{1}{T} \langle {\mathcal M} \rangle_T \le a\right]\\
=& {\mathbb P}\left[ \frac{1}{T}{\mathcal M}_T\ge \frac{\varepsilon}{2} \right]\le \exp \left( -\frac{T \varepsilon^2}{8(a + b \varepsilon/6)} \right)
\end{align*}
with any $\varepsilon>0$. Note that for any $\delta \in (0,1) $,
$\exp \left(-\frac{T\varepsilon^2}{8(a + b\varepsilon /6)}\right)\le \delta$ holds if and only if
$T \ge \frac{8(a + b\varepsilon/6)}{\varepsilon^2}\ln(\delta^{-1})$. By plugging $a$ and $b$ given before into the last inequality, it holds that if
\begin{align*}
&T\ge \frac{8C^2 ((\alpha_0+2)^2 C^2+(\alpha_0+4)\varepsilon/6)}{\varepsilon^2}\ln \left(\frac{1}{\delta}\right),
\end{align*}
then with probability at least $1-\delta$, we have ${\mathcal M}_T /T \le\varepsilon/2$. Combined with~\eqref{eq:8}, one concludes that under the conditions in the statement of~\cref{prop:convergence1}, with probability at least $1-\delta$, ${\mathcal L}(\hat x_T,w)-{\mathcal L}(x,\hat w_T)\le \varepsilon$ holds. By~\cref{def:e-saddle-point}, it implies
\begin{align*}
&{\mathbb P}[(\hat x_T,\hat w_T ) \in {\mathcal H}_\varepsilon]\ge 1-\delta.
\end{align*}
This completes the proof.

\end{document}